 \documentclass[11pt]{amsart}
\usepackage[
	hypertexnames=false,
	hyperindex,
	pagebackref,  
	pdftex,
	breaklinks=true,
	bookmarks=false,
	colorlinks,
	linkcolor=blue,
	citecolor=red,
	urlcolor=red,
]{hyperref}  

\usepackage{xcolor}
  
\newtheorem{theorem}{Theorem}[section]
\newtheorem{lemma}{Lemma}[section]
\newtheorem{proposition}{Proposition}[section]

\theoremstyle{definition}
\newtheorem{definition}{Definition}[section]
\theoremstyle{remark}
\newtheorem{remark}{Remark}[section]
\numberwithin{equation}{section}
\newtheorem{corollary}{Corollary}[section]

\numberwithin{equation}{section}

\newcommand{\M}{{\mathcal M}}
\newcommand{\C}{\mathbb C}
\newcommand{\Z}{\mathbb Z}
\newcommand{\R}{\mathbb R}
\newcommand{\Out}{{\rm Out}}
\newcommand{\Aut}{{\rm Aut}}

\begin{document}

\title[Braided surfaces]{Braided surfaces and their characteristic maps}
\author[L.Funar]{Louis Funar}
\author[P.G.Pagotto]{Pablo G. Pagotto}
\address{Institut Fourier BP 74, UMR 5582, Laboratoire de Math\'ematiques, Universit\'e Grenoble Alpes, CS 40700, 38058 Grenoble cedex 9, France}
\email{louis.funar@univ-grenoble-alpes.fr}
\email{pgp\underline{  }2008@hotmail.com}

\begin{abstract}
We show that branched coverings of surfaces of large enough genus arise as 
characteristic maps of braided surfaces that is, lift to embeddings in the product of the surface with $\R^2$.
This result is nontrivial already for unramified coverings, in which case the lifting problem is well-known to reduce to the purely algebraic problem of factoring the monodromy map to the symmetric group $S_n$ through the braid group $B_n$. In our approach, this factorization is often achieved as a consequence of a stronger property: a factorization through a free group. 
In the reverse direction we show that any non-abelian surface group 
has infinitely many finite simple non-abelian groups quotients
with characteristic kernels which do not contain any simple loop and 
hence the quotient maps do not factor through free groups. 
By a pullback construction, finite dimensional Hermitian representations of braid groups provide 
invariants for the braided surfaces. We show that 
the strong equivalence classes of braided surfaces are separated by such invariants if and only if they are profinitely separated. 
 
\vspace{0.2cm}
 
\noindent MSC Class: 57R45, 57 R70, 58K05.

\end{abstract}

\maketitle
\section{Introduction}
The question addressed in the present paper is the description of 
a particular case of 2-dimensional knots, called braided surfaces, up to fiber preserving isotopy. 
\begin{definition}
A {\em braided surface} over some surface $\Sigma$ is an embedding of a surface 
$j:S\to  \Sigma\times \R^2$, such that the composition 
\[S\stackrel{j}{\hookrightarrow} \Sigma\times \R^2\stackrel{p}{\to} \Sigma\]
with the first factor projection $p$ is a branched covering. Throughout this 
paper we only consider {\em locally flat PL} embeddings $j$. 
The composition $p\circ j$ is called the {\em characteristic map} of the braided surface $S$.  
\end{definition}

Braided surfaces over the disk were first considered  and studied by Viro, 
Rudolph (see \cite{Rudolph83}) and later extensively studied by Kamada (\cite{Kamada}). 
A comprehensive survey of the subject could be found in the monograph \cite{Kamada2}.  
Braided surfaces over the torus were introduced more recently in \cite{Nakamura}. 

Equivalence classes of braided surfaces are in one-to-one correspondence 
with a subset of the set of representations of the punctured surface group 
into the braid group up to conjugacy and mapping class group action. 
Our aim is to give some insight about the structure of such discrete representation varieties. 

\begin{definition}
A map $S\to \Sigma$ between surfaces is 
called a {\em 2-prem} if it factors as above as $p\circ j$, where $j$ is an embedding 
and $p$ is the second factor projection $\Sigma\times \R^2\to \Sigma$. 
\end{definition}
Whether all generic smooth maps are 2-prems seems  widely open.   
We refer to the article \cite{Melikhov} of Melikhov  for the state of the art on this question. 
Although branched coverings are not generic the question of whether they 
are 2-prems seems natural.  Our first result gives an affirmative answer in the asymptotic range: 

\begin{theorem}\label{stable}
There exists some $h_{n,m}$ such that any degree $n$ ramified covering of 
the closed orientable surface of genus $g\geq h_{n,m}$ with $m$ branch points occurs as 
a 2-prem. 
\end{theorem}

The key ingredient is the description of the mapping class group orbits on 
the space of surface group representations onto a finite group in the stable range, i.e. for large genus $g$. This was done by Dunfield and Thurston (see \cite{DT}) in the closed case and by 
Catanese, L\"onne and Perroni (see \cite{CLP2}) in the branched case. Their results establish
the classification of these orbits by means of some versions of homological Schur invariants. 
Note that the bound $h_{n,m}$ is not explicit.

The genuine classification of these orbits seems much subtler, see 
\cite{Lubo} for a survey of this and related questions. Livingston provided (\cite{Liv1})
examples of distinct orbits with the same homological Schur invariants.

\begin{definition}
The homological {\em Schur invariant} of a 
surjective homomorphism $f:\pi_1(\Sigma)\to G$ of a closed orientable surface group onto a group $G$ 
is the image  $sc(f)\in H_2(G)$  of the fundamental class $[\Sigma]\in H_2(\Sigma)$ by $f$. 
Recall that  such a homomorphism $f:\pi_1(\Sigma)\to G$  is called  {\em elementary}
if it factors through a free group. 
\end{definition}

The null-homologous case, namely where the homological Schur invariant vanishes, corresponds to finding 
whether a surjective homomorphism $f:\pi_1(\Sigma)\to G$ of a closed surface group onto a finite group $G$ 
inducing a trivial map in 2-homology is  elementary. This amounts to estimating the minimal Heegaard 
genus for a 3-manifold group to which $f$ extends, problem which was recently considered by Liechti and 
March\'e for tori (see \cite{LM}). 

These questions arose in relation with the equivalence problem for  
epimorphisms of free groups onto non-abelian simple groups. 
Let $\mathbb F_n$ denote the free group on $n\geq 3$ generators and $\Out(\mathbb F_n)$ its 
{\em outer automorphism} group. 
Wiegold conjectured (see \cite{Lubo}) that for any finite simple non-abelian group $G$ the group  
$\Out(\mathbb F_n)$ acts transitively on the set of conjugacy classes of surjective homomorphisms $\mathbb F_n \to G$.
A weaker statement which allows additional stabilizations  is known to hold (see \cite{McW}). Also 
Gilman (\cite{Gilman}) and Evans  (\cite{Evans,Evans2})  proved that there exists a large orbit of $\Out(\mathbb F_n)$ on  
this set. In  \cite{BL,DT} the authors proved that  the action 
of the mapping class groups $\Gamma(\Sigma)$ on the set of conjugacy classes of surjective homomorphisms 
onto finite groups $G$ also has at least one large orbit.  

Wiegold's conjecture implies that there is no isolated orbit, namely there is no finite simple quotient $G$ which is characteristic. Recall that a subgroup $H\subseteq G$  
is a {\em characteristic}  subgroup of $G$ if it is invariant by all automorphisms of $G$. In this case  
$G/H$ is called a {\em characteristic quotient} of $G$.

In  \cite{FL} one proved that there exist  finite simple non-abelian quotients of surface groups which are 
characteristic, by using quantum representations.  Conjugacy classes of surjective homomorphisms onto characteristic quotients 
of $\pi_1(\Sigma)$ are therefore isolated orbits  for the action of the mapping class group, contrasting 
with large orbits from \cite{BL,DT}.  
An easy consequence is that all these quotient epimorphisms are 
non-elementary, so that the classification 
of mapping class group orbits fundamentally differs from the stable one. This improves previous work of 
Livingston (\cite{Liv1,Liv2}) and Pikaart (\cite{Pikaart})(see Propositions \ref{nonelementary} and \ref{quantum}): 

\begin{theorem}\label{simplequotients}
For any $g\geq 2$ there exist infinitely many simple non-abelian groups $G$ and 
surjective homomorphisms of the closed genus $g$ orientable surface  onto $G$, such that  
the kernels are characteristic and do not contain any simple loop homotopy class. 
In particular, these homomorphisms are not elementary. 
\end{theorem}

\begin{remark}
Given  an embedding $G\subset S_n$, if the surjective homomorphism $\pi_1(\Sigma)\to G$ is elementary, 
then $f$ can be lifted to $\pi_1(\Sigma)\to B_n$. We don't know whether  the non-elementary 
homomorphisms from Theorem \ref{simplequotients} admit lifts to $B_n$, see also Remark \ref{indepn}. 
\end{remark}

\begin{definition}
Two braided surfaces $j_i:S\to \Sigma\times \R^2$, $i=0,1$ over $\Sigma$ are {\em (Hurwitz) equivalent} if there exists some ambient isotopy $h_t:\Sigma\times \R^2\to \Sigma\times \R^2$, $h_0=id$ such that 
$h_t$ is fiber-preserving and $h_1\circ j_0=j_1$. Recall that $h_t$ is {\em fiber-preserving} if there exists a homeomorphism $\varphi_t: \Sigma\to \Sigma$ such that $p\circ h_t=\varphi_t\circ p$. There is no loss of generality to impose 
$\varphi_t$ to be an isotopy of $\Sigma$. Assume that the branch loci of the branched coverings $j_i\circ p$ are the same finite set $B$.  When $\varphi_t$ can be taken to be isotopy fixing pointwise the branch locus $B$, we say that the braided surfaces are {\em strongly equivalent}. These definitions extend naturally to the case when these surfaces have boundary by requiring isotopies to fix the boundary points of $j_i(S)$. 
\end{definition}

In the last part of this paper we show that finite dimensional Hermitian representations of braid groups provide invariants for the strong equivalence classes of braided surfaces, by a standard pullback construction (see section \ref{invariants}), called spherical functions. 
We then show that the topological information underlying the spherical functions 
is of profinite nature (see Theorem \ref{profinite} for the general statement): 

\begin{theorem}\label{separatebraid}
Strong equivalence classes of braided surfaces are separated by some spherical function if and only if 
they are profinitely separated. 
\end{theorem}

\noindent{\bf Acknowledgements.} The authors are grateful to P. Bellingeri, G. Kuperberg, 
L. Liechti, M. L\"onne, J. March\'e,  J.B. Meilhan, S. Melikhov, E. Samperton and E. Wagner for useful conversations and to the referees for pointing out several errors and incomplete 
arguments in the previous versions of this paper and improving the exposition.

\section{Braided surfaces} 
Consider a braided surface over  a closed orientable surface  $\Sigma$, namely  a locally flat PL embedding of a 
closed orientable surface  $j:S\to  \Sigma\times \R^2$, such that the composition  $p\circ j$  is a branched covering. 
We might consider, more generally, that $S$ is embedded in a (orientable) plane bundle 
over $\Sigma$. However, the existence of nontrivial examples, for instance that  
some connected unramified covering of degree $>1$ arise as a characteristic map, 
implies that the plane bundle should be trivial (see \cite{Edmonds2}).

A degree $n$ branched covering $S\to \Sigma$ determines a homomorphism 
$f:\pi_1(\Sigma\setminus B,*)\to S_n$, where $B$ is the branch locus of $F$, called 
{\em monodromy} homomorphism. Choose small simple loops 
$\gamma_i$ each one encircling one branch point $b_i$ of $B$, which will be called 
{\em peripheral} loops or homotopy classes in the sequel. 

The {\em degree} $n$ and {\em branch locus} $B$ of a braided surface $j:S\to  \Sigma\times \R^2$ is 
are, respectively, the degree and branch locus of its associated branched covering map $p\circ j$. 
Observe that the projection map  
\[ p|_{(\Sigma\times \R^2)\setminus j(S)}:  (\Sigma\times \R^2)\setminus j(S)\to \Sigma\] 
restricts to a locally trivial fiber bundle over $\Sigma-B$. The monodromy of this 
locally trivial fiber bundle is then a homomorphism 
\[ f:\pi_1(\Sigma\setminus B)\to B_n\]
into the braid group $B_n$ on $n$-strands, which will be called the {\em braid monodromy} of the braided surface in the sequel.

Recall that two branched coverings $F_1, F_2:S\to \Sigma$ are 
{\em equivalent} if there are homeomorphisms $\Phi:S\to S$ and $\phi:\Sigma\to \Sigma$ 
such that $F_1\circ \Phi=\phi\circ F_2$. They are further {\em strongly equivalent}  when there is some $\phi$ which is isotopic to the identity rel the branch locus.

Hurwitz proved that   
strong equivalence classes of branched coverings with given  $g$ genus of $\Sigma$, $B$ and $n$, bijectively correspond to conjugacy classes of monodromy homomorphisms having nontrivial image on every peripheral loop. 
Denote by $\Gamma(\Sigma)$ the pure {\em mapping class group} of the possible punctured surface $\Sigma$.  
Moreover, equivalence classes of branched coverings bijectively correspond to orbits of 
the mapping class group $\Gamma(\Sigma\setminus B)$ on the set of conjugacy classes of monodromy homomorphisms
as above.

We will show that a similar result holds in the case of braided surfaces.  
Let $\gamma\subset \Sigma\setminus B$ be an embedded loop. Its preimage 
$\ell_{\gamma}=p^{-1}(\gamma)\cap j(S)$ is a link in the open solid torus $p^{-1}(\gamma)\simeq \gamma\times \R^2$. The link $\ell_{\gamma}$ is the link closure $\widehat{b}$ of a braid 
$b\in B_n$ within the solid torus, because the projection $p|_{\ell_{\gamma}}:\ell_{\gamma}\to \gamma$ is an unramified covering. 
Note that the link $\ell_{\gamma}$ in the solid torus determines and is determined 
by the conjugacy class of $b$ in $B_n$. 

If $\gamma$ were a peripheral loop, let us choose a bounding disk $\delta$ embedded 
in $\Sigma$, which contains a single point of $B$. Since $S$ is compact, we can assume that $j(S)\subset \Sigma\times D^2$, where 
$D^2\subset \R^2$ is a compact disk. 
Then $p^{-1}(\delta)\cap \Sigma\times D^2\simeq \delta \times D^2$  is a manifold with corners 
diffeomorphic (after rounding the corners) with the 4-disk $D^4$.  
In particular, the solid torus link $\ell_{\gamma}$ determines a link in $S^3$ by means of the   
embedding $\ell_{\gamma}\subset \gamma\times D^2\subset \partial D^4$. 

\begin{definition}
A solid torus link $\ell\subset \gamma\times D^2$ is {\em completely split} if 
there exist  disjoint disks $D_i^2\subset D^2$ such that each connected component of $L$ is 
contained within a solid torus $\gamma\times D_i^2$. 
The braid $b\in B_n$ is {\em completely splittable} if the corresponding 
link  $\widehat{b}\subset \gamma\times D^2$ is completely split as a solid torus link 
and also trivial as a link in $S^3$.   
We denote by $\mathcal A_n\subset B_n\setminus\{1\}$ the set of completely splittable nontrivial braids.   
\end{definition}

If $\ell\subset \gamma\times D^2$ is a completely split unlink with components $\ell_i$, then choose points $x_i\in D_i^2$ and let $y$ be the single branch point belonging to ${\rm int}(\delta)$. Let $C(\ell_i)$ be the cone on $\ell_i$ with vertex $(y,x_i)\in \delta\times D^2$ and $C(\ell)$ be the union of $C(\ell_i)$. Since each component $\ell_i$ of $\ell$ is a trivial knot in $\partial(\delta\times D^2)$, the multi-cone $C(\ell)$ is the disjoint union 
of locally flat embedded disk in $\delta\times D^2$. Note that 
$C(\ell)$ is a braided surface over the disk $\delta$ with a single branch point $\{y\}$. 
By \cite{Kamada2}, Lemma 16.11) this is the unique braided surface over $\delta$ with 
branch point $\{y\}$ and boundary $\ell$.

\begin{lemma}\label{nobranching}
A braided surface $S$ of degree $n$ over $\Sigma$ without branch points is determined up to 
equivalence rel boundary by its  braid monodromy homomorphism 
$f:\pi_1(\Sigma)\to B_n$. 
\end{lemma}
\begin{proof}
A homomorphism $f$ corresponds to a  unique locally trivial fiber bundle  over $\Sigma$ with 
fiber $D^2\setminus\{p_1,p_2,\ldots,p_n\}$ which is trivialized on the boundary $\partial D^2$ fiber bundle. 
One constructs the braided surface over a wedge of circles first and observe that it extends over the 2-cell 
which produces the surface $\Sigma$, as $f$ is a group homomorphism. 
\end{proof}

\begin{theorem}\label{liftingbraid}
A homomorphism $f: \pi_1(\Sigma\setminus B)\to B_n$ arises as the braid monodromy 
of some braided surface of degree $n$ with branch locus $B$ if and only if 
$f$ sends each peripheral loop $\gamma_i$ into a completely splittable nontrivial braid $f(\gamma_i)\in \mathcal A_n\subset B_n$. Moreover, a braided surface $S$ of degree $n$ over $\Sigma$ is determined up to strong equivalence by its braid monodromy $f: \pi_1(\Sigma\setminus B)\to B_n$.  
\end{theorem}
\begin{proof}
Consider disjoint disks $\delta_i$ bounded by peripheral loops $\gamma_i$, for all branch points and let $X$ be their complement. Then $j(S)\cap p^{-1}(\delta_i)$ is a 
braided surface over the disk $\delta_i$. By (\cite{Kamada2}, Lemma 16.12) 
$j(S)\cap p^{-1}(\delta_i)$ has a braid monodromy homomorphism $f$ with $f(\gamma_i)$ completely splittable. This proves the necessity of our conditions. 

Conversely, the multi-cone $C(\ell_{\gamma_i})$ over the link $\ell_{\gamma_i}$ 
provides a braided surface over $\delta_i$. 
The homomorphism $f:\pi_1(\Sigma-\cup\delta_i)\to B_n$ provides by Lemma \ref{nobranching} a  
unique embedding $j:S'\to \Sigma\times \R^2$ which has no branch points. 
We then glue together $S'$ and the cones $C_i$ along their boundaries, in order to respect 
the projection map $p$. As the glued surface $S$ is unique, the braided surface 
is determined up to strong equivalence by $f$ (see also \cite{Kamada} and \cite[Thm. 17.13]{Kamada2}). 
\end{proof}

As an immediate consequence we have: 

\begin{corollary}\label{corliftable} 
The degree $n$ branched covering $S\to \Sigma$ with branch locus $B$  
is the characteristic branched covering of a braided surface over $\Sigma$ 
if and only if its monodromy homomorphism $f:\pi_1(\Sigma\setminus B)\to S_n$ can be lifted to a homomorphism $F:\pi_1(\Sigma\setminus B)\to B_n$ such that $F$ sends peripheral 
loops into completely splittable nontrivial braids. 
\end{corollary}
\begin{proof}
Theorem \ref{liftingbraid} yields a braided surface lifting some branched covering $S'\to \Sigma$ of degree $n$, 
branch locus $B$ and the prescribed monodromy $f$. As the ramification degrees are determined by the cycle structure of the permutations corresponding to the peripheral loops this branched covering can be identified with $S\to \Sigma$. 
\end{proof}

When $B=\emptyset$ we retrieve the following result of Hansen (\cite{Hansen,Hansen2}): 

\begin{corollary}
The degree $n$ unramified covering $S\to \Sigma$ factors as the composition
\[S\stackrel{j}{\hookrightarrow} \Sigma\times \R^2\stackrel{p}{\to} \Sigma\]
of some embedding $j$ and the second factor projection $p$, if and only if 
its monodromy map $f:\pi_1(\Sigma)\to S_n$ 
lifts to a homomorphism $\pi_1(\Sigma)\to B_n$.
\end{corollary}

Another consequence is
 
\begin{corollary}
Degree $n$ braided surfaces on $\Sigma$ with branch locus $B$, up to strong equivalence rel boundary are 
in one-to-one correspondence with the set $M_{B_n}(\Sigma,B)$ of homomorphisms 
$\pi_1(\Sigma\setminus B)\to  B_n$ sending peripheral loops into $\mathcal A_n$ modulo 
the conjugacy action by $B_n$. Furthermore, the classes of these 
braided surfaces up to equivalence rel boundary are in one-to-one with the 
set  $\mathcal M_{B_n}(\Sigma, B)$ of orbits of the mapping class group $\Gamma(\Sigma\setminus B)$ 
action on  $M_{B_n}(\Sigma,B)$ by left composition.  In particular, braided surfaces 
provide a topological interpretation for the space of double cosets 
$B_{\infty}\backslash B_{\infty}^k/B_{\infty}$, studied by Pagotto in \cite{Pagotto1,Pagotto2}.   
\end{corollary}

\begin{remark}\label{indepn}
Symmetric groups and braid groups form nested sequences $\subset S_n\subset S_{n+1}\subset \cdots$ 
$\subset B_n\subset B_{n+1}\subset \cdots$, where inclusions are induced by adding one more strand on the right. 
Inclusions are compatible with the projections $p_n:B_n\to S_n$. We note that the answer to the lifting question for homomorphism $f:\pi_1(\Sigma\setminus B)\to S_n$ is independent on the chosen value for $n$. This follows from the existence 
of a group homomorphism $p_{n+1}^{-1}(S_n)\to B_n$ induced by the map removing the last strand from the right, which 
sends completely splittable braids into completely splittable braids.  
\end{remark}

\begin{remark}
The braided surfaces whose characteristic branched covering is a simple branched covering 
are analogous to achiral Lefschetz fibrations. The monodromy around a branch point is given by a 
band, namely a standard generator of the braid group or its inverse. 
\end{remark}


\begin{remark}
Recovering braided surfaces from their characteristic maps is just an instance of more general 
questions about compactifications of fibre bundles. There are examples of smooth maps 
between closed manifolds in specific dimensions having only finitely many critical points 
(see e.g. \cite{F11}). Characterizing the fibre  bundle arising in the complementary of the 
critical locus and how they determine the original maps might have far-reaching implications. 
\end{remark}

\section{Lifting homomorphisms and the proof of Theorem \ref{stable}}
\subsection{The stable lifting problem}
A basic problem in algebra and topology is, for a given  
surjective homomorphism  $p: \tilde{G}\to G$, to characterize those group homomorphisms
$f:J\to G$  which admit a lift to $\tilde{G}$, namely a homomorphism 
$\varphi:J\to \tilde{G}$ such that $p\circ\varphi=f$.
In the simplest case when $J$ is a free group any homomorphism is liftable.
The next interesting case is $J=\pi_1(\Sigma_g)$, where 
$\Sigma_g$ denote the genus $g$ closed orientable surface and $g\geq 2$.  
The lifting question might appear under a slightly more general form, by requiring 
that $(\varphi(\gamma_i))_{i=1,\ldots,m}=(c_i)_{i=1,\ldots,m}\in \widetilde{G}$, for a set of elements $\gamma_i\in J$, $c_i\in \widetilde{G}$. 

Let $\Sigma'$ be a closed orientable surface and $\Sigma$ a surface, possibly punctured. Denote by 
$\Sigma\sharp \Sigma'$ the connected sum. There is a natural map  $\Sigma\sharp \Sigma'\to \Sigma$, called {\em pinch} which consists of crushing the complement of an open disk in $\Sigma'$ to a point. The operation which replaces $\Sigma$ by  $\Sigma\sharp \Sigma'$ will be called a (genus) {\em stabilization}.

Although in general it seems difficult to lift homomorphisms $f$ (see \cite{Melikhov,Petersen}) there is only a homological obstruction to lift $f$, if we allow the surface be stabilized, as it will be explained below.
  
Let $\Sigma_h\setminus B$ be a stabilization of the surface $\Sigma_g\setminus B$, $h\geq g+1$ and 
let $P:\pi_1(\Sigma_{h}\setminus B)\to \pi_1(\Sigma_g\setminus B)$ be the homomorphism induced by the pinch map. 
If $f:\pi_1(\Sigma_g\setminus B)\to G$ is a homomorphism, we 
call the composition $f\circ P$ a   (genus) {\em  stabilization} of $f$. 
We further say that $f:\pi_1(\Sigma_g\setminus B)\to G$ {\em stably lifts} along $p:\tilde{G}\to G$ if it has some stabilization $f'=f\circ P:\pi_1(\Sigma_h\setminus B)\to G$ 
which lifts to $\tilde{G}$.

\subsection{Lifting in the unramified case} We start with an outline of the proof of Theorem \ref{stable} in the
unramified case. A homomorphism $\pi_1(\Sigma_g)\to G$ corresponds to a homotopy class of based 
maps $f:\Sigma_g\to K(G,1)$, thereby defining the Schur class $sc(f)=f_*([\Sigma_g])\in H_2(G)$.

Recall that two surjective homomorphisms $f,f':\pi_1(\Sigma_g)\to G$ are {\em equivalent} if there 
exists an automorphism  $\Theta\in \Aut^+(\pi_1(\Sigma_g))$ such that $f'=f\circ \Theta^{-1}$. 
Here $\Aut^+(\pi_1(\Sigma_g))$ is the group of automorphisms of the fundamental group which are induced by 
homeomorphisms preserving the orientation and fixing a point of the surface $\Sigma_g$. Alternatively, these are 
those automorphisms of  $\pi_1(\Sigma_g)$ which act trivially on $H_2(\pi_1(\Sigma_g))$. 
Now, Zimmermann  (\cite{Zim}, see also \cite{Liv}) proved that  group epimorphisms have stabilizations which are equivalent if and only if their classes in the second homology agree. 

This implies that an epimorphism stably lifts to $\tilde G$ if and only if its Schur  
class in $H_2(G)$ lies in the image of $H_2(\tilde{G})$. Indeed every class in $H_2(\tilde G)$ is the Schur class of some 
homomorphism $\pi_1(\Sigma_{\tilde g})\to \tilde G$, and moreover it is not hard to find a $\tilde g$ and  
such a homomorphism which is surjective (see Lemma \ref{surjective}). 

Dunfield and Thurston (see \cite{DT}) improved this result when the group $G$ is finite. They showed that 
there exists $g(G)$ with the property that any two surjective homomorphisms 
$f, f':\pi_1(\Sigma_g) \to G$ with $g\geq g(G)$ having the  same  Schur class in $H_2(G)$ 
are  already equivalent under the action of $\Gamma(\Sigma_{g,1})\times G$, where $G$ acts by inner automorphisms by right composition. 
The same argument as above shows that for large enough $g\geq g(G)$ a  
surjective homomorphism $f:\pi_1(\Sigma_g) \to G$ lifts to $\tilde G$ if and only if 
$sc(f)$ lies in the image of $H_2(\tilde G)$. Eventually, when $G\subseteq S_n$ and $\tilde G$ is the preimage of $G$ within the braid group $B_n$, one shows that $H_2(\tilde G)\to H_2(G)$ is surjective (see Lemma \ref{H_2surjective}). This proves our claim. 
 
 The rest of this section is devoted to make this strategy work for the ramified case as well.

\subsection{Schur invariants for punctured surfaces}
We now describe a construction of homological Schur invariants for homomorphisms 
$\pi_1(\Sigma_g\setminus B)\to G$. At first let $D^2$ be a disk embedded in $\Sigma_g$ containing the punctures 
$B=\{b_1,b_2,\ldots,b_m\}$ and $\gamma_i$ be a based loop encircling once the puncture $b_i$, so that $\gamma_i$ are 
pairwise disjoint except for their base-point. Identify then $\Sigma_g\setminus B$ with the boundary union of 
$\Sigma_g\setminus D^2$ and $D^2\setminus B$, so that there is a fixed system of curves $\gamma_i$ 
on $\Sigma_g\setminus B$.

Consider the surface with boundary $\Sigma$ obtained from $\Sigma_g$ after 
removing pairwise disjoint open small disks around each puncture 
$b_i$, namely replacing the puncture $b_i$ with a boundary component $\mathbf b_i$. 
Let also $\Sigma^{\circ}$ be the result of cutting $\Sigma$ along the curves 
$\gamma_i$ and discarding the annuli bounded by $\mathbf b_i$ and $\gamma_i$. 

Given the elements $\mathbf c=(c_i)_{i=1,\ldots,m}\in G^m$ we represent them as  homotopy 
classes of  based oriented loops $\ell_i$ embedded within 
the space $K'(G,1)=K(G,1)\times \R^5$, which are disjoint except for their base-point.
Let also $L_i$ be disjoint embedded oriented loops in $K'(G,1)$ such that each pair 
$\ell_i$ and $L_i$ bounds an embedded annulus $A_i$ in $K'(G,1)$.    
Let $L_{\mathbf c}$ be the union of $L_i$.
 
A homomorphism $f:\pi_1(\Sigma_g\setminus B)\to G$ such that 
$f(\gamma_i)=c_i\in G$, for every $i$, provides a continuous 
based map $\phi:\Sigma^{\circ}\to K'(G,1)$, which is unique up to homotopy. 
Then $\phi(\gamma_i)$  is based homotopic to $\ell_i$. By adjoining these homotopies 
we can arrange that  $\phi(\gamma_i)=\ell_i$. By gluing the annuli $A_i$ 
we obtain a  based map $\phi: \Sigma\to K'(G,1)$ which sends 
$\partial \Sigma$ homeomorphically onto $L_{\mathbf c}$. 
Two based homotopies between $\phi(\gamma_i)$  and $\ell_i$ define a map from 
a 2-sphere (with poles identified) into $K'(G,1)$ which must extend to the 3-disk, since 
$\pi_2(K'(G,1))=0$. This implies that 
\[ \phi_*([\Sigma,\partial \Sigma])\in H_2(K'(G,1), L_{\mathbf c})\]
is a well-defined homology class in the relative homology, independent on the 
various choices made in the construction.

Of course a homomorphism $f$ as above could only exist if $\prod_{i=1}^m c_i$ belongs to the 
commutator subgroup $[G, G]$, which we assume to be the case from now on. 
\begin{definition}
Let $\mathbf c=(c_i)_{i=1,\ldots,m}\in G^m$ with  $\prod_{i=1}^m c_i\in [G,G]$, and choose  
a system of curves $\gamma_i$ and a link $L_{\mathbf c}$ as above. 
We denote by $H_2(G; \mathbf c)$ the group $H_2(K'(G,1), L_{\mathbf c})$ and say that 
 $sc(f)=\phi_*([\Sigma,\partial \Sigma])\in H_2(G; \mathbf c)$ is the {\em Schur class} of $f$. 
\end{definition}
From the exact sequence of the pair $(K'(G,1), L_{\mathbf c})$ 
we derive the exact sequence:  

\[0 \to H_2(G)\to H_2(G, {\mathbf c})\to \Z^m\to H_1(G)\]
As the map $\phi$ is a degree one map on the circles $\gamma_i$, the image of 
$sc(f)$ in $H_1(L_{\mathbf c})=\Z^m$ is $(1,1,\ldots,1)$ and hence the rightmost map sends 
$(1,1,\ldots,1)\in \Z^m$ into $0\in H_1(G)$, by exactness. Classes 
in $H_2(G,{\mathbf c})$ whose image is $(1,1,\ldots,1)\in \Z^m$  will be called {\em primitive}.  

A similar invariant, denoted $\varepsilon(f)$, was defined by Catanese, L\"onne and Perroni in \cite{CLP1}. Our invariant $sc(f)$ is non-canonical, in the sense that it depends on the choice 
of the curves $\gamma_i$ and $L_{\mathbf c}$, while $\varepsilon(f)$  is canonical. 
The construction of $\varepsilon(f)$ proceeds as above, working with all possible values of 
$\mathbf c$ at once. The target group  in \cite{CLP1} would naturally be 
$H_2(K(G,1),K(G,1)^{(1)})$, where $K(G,1)^{(1)}$ is the 1-skeleton of $K(G,1)$.  
However, the classes so obtained in  $H_2(K(G,1),K(G,1)^{(1)})$ are only well-defined when 
we pass to a suitable quotient of it identifying classes of surfaces $\Sigma$ whose 
boundaries are only freely homotopic in $K(G,1)$.

This equivalence relation between surface groups homomorphisms readily extends to surjective homomorphisms 
$f:\pi_1(\Sigma_g\setminus B)\to G$ with prescribed peripheral monodromy $f(\gamma_i)=c_i$, for $i=1,\ldots,m$. 
Then two homomorphisms $f$ and $f'$ as above are equivalent if there exists some 
$\Theta\in {\rm SAut}^+(\pi_1(\Sigma_g\setminus B))$ such that $f'=f\circ \Theta^{-1}$. Here 
${\rm SAut}^+(\pi_1(\Sigma_g\setminus B))$ denotes the group of automorphisms of the group $\pi_1(\Sigma_g\setminus B)$ 
which are induced by homeomorphisms preserving the orientation of $\Sigma_g\setminus B$ fixing a point of the surface and 
preserving pointwise the punctures along with the peripheral monodromy, that is 
$f\circ \Theta^{-1}(\gamma_i)=c_i$, for  $i=1,\ldots,m$. Observe that ${\rm SAut}^+(\pi_1(\Sigma_g\setminus B))$ contains  the automorphisms of $\pi_1(\Sigma_g\setminus B)$ whose classes belongs to the subgroup $\Gamma(\Sigma_{g,1}) \subset \Gamma(\Sigma_g\setminus B, *)$ of those mapping classes of homeomorphisms which are the identity on $D^2\setminus B\subset \Sigma_g\setminus B$.  Then the Schur class  $sc(f)\in H_2(G, \mathbf c)$ of a homomorphism $f$ is invariant with respect to the 
left action by ${\rm SAut}^+(\pi_1(\Sigma_g\setminus B))$. 

When $B=\emptyset$ the equivalence relation is compatible with the 
$G$-conjugacy. Consider the set 
\[ M_G(\Sigma_g)={\rm Hom^s}(\pi_1(\Sigma_g), G)/G\]
of $G$-conjugacy classes of {\em surjective} homomorphisms $f$. 
There is an obvious action of the mapping class group $\Gamma(\Sigma_g)$ on 
$M_G(\Sigma_g)$ by left composition. 
We say that $G$-conjugacy classes of homomorphisms are equivalent if they belong to the same
$\Gamma(\Sigma_g)$-orbit. Conjugacy in $G$ acts trivially on $H_2(G)$. 
Livingston  (\cite{Liv}, see also \cite{Zim}) has 
proved that $G$-conjugacy classes of surjective homomorphisms are stably equivalent 
if and only if their classes in $H_2(G)$ agree.

In the punctured case we fix an $m$-tuple $\mathbf c\in G^m$ and its 
conjugacy class with respect to the diagonal action: 
\[G\cdot {\mathbf c}=\{(a c_i a^{-1})_{i=1,\cdots,m}, \ |  a\in G\}\subset G^m.\] 
We then consider the set of surjective homomorphisms mod conjugacy: 
\[ M_G(\Sigma_g, B, {\mathbf c})=\{f\in {\rm Hom^s}(\pi_1(\Sigma_g\setminus B), G) |  (f(\gamma_i))_{i=1,\cdots,m}\in G\cdot \mathbf c\}/G\]
where ${\rm Hom^s}$ denotes the surjective homomorphism. 
The {\em pure} mapping class group $\Gamma(\Sigma_g\setminus B)$ (which fixes the punctures $b_i$ 
pointwise) has a left action on  ${\rm Hom}(\pi_1(\Sigma_g\setminus B), G)/G$ which keeps the subspace $M_G(\Sigma_g,B,\mathbf c)$ invariant. Conjugacy classes are said equivalent if they 
determine the same element in the orbit set:  
\[ \mathcal M_G(\Sigma_g,B,{\mathbf c}) = \Gamma(\Sigma_g\setminus B) \backslash M_G(\Sigma_g,B,\mathbf c).\] 
Observe that the conjugacy by $a\in G$ sends isomorphically $H_2(G;\mathbf c)$ onto 
$H_2(G, a\mathbf c a^{-1})$, in particular $sc(f)$ does not descends 
to $M_G(\Sigma_g, B, {\mathbf c})$.
However we can identify those pairs of elements in the union of groups $H_2(G;\mathbf b)$, 
where $\mathbf b\in G\cdot \mathbf c$ which are related by some conjugacy isomorphism. 
The result is a quotient of $H_2(G;\mathbf c)$ which was explicitly described 
by Catanese, L\"onne and Perroni in \cite{CLP1}. Moreover, the image of $sc(f)$ in 
this quotient group is the same as their $\varepsilon$ invariant which is defined 
on $M_G(\Sigma_g, B, {\mathbf c})$.

\subsection{Stable equivalence for punctured surfaces}
The previous result of Livingston and Zimmermann on $G$-conjugacy classes 
was extended to the punctured case by Catanese, L\"onne and Perroni in \cite{CLP2}. 
Specifically, the $G$-conjugacy classes  from $M_G(\Sigma_g, B, {\mathbf c})$ are stably equivalent if and only if their $\varepsilon$-invariants agree. 
There is a corresponding result for genuine homomorphisms, as follows:  
 
\begin{proposition}\label{Livingston}
Surjective homomorphisms of surface groups with the same puncture set $B$  
and boundary holonomy $\mathbf c\in G^m$ are stably equivalent if and only if 
their Schur classes in $H_2(G,{\mathbf c})$ agree.  
\end{proposition}
\begin{proof}
One can derive this from the corresponding stability result in \cite{CLP2}. 
However, there is a direct proof following the lines of the closed case (see \cite{Liv}). 
First, the class $sc(f)$ is preserved by stabilizations. 
Further, if $\Omega_n(X,A)$ is the dimension $n$ orientable bordism group associated to the pair $(X,A)$ of CW complexes, then seminal work of Thom implies that the natural homomorphism 
\[ \Omega_n(X,A)\to H_n(X,A)\]
is an isomorphism if $n\leq 3$ and an epimorphism if $n\leq 6$ (see e.g. \cite{Rudyak}, Thm. IV.7.37). In particular, the classes in $H_2(G, p({\mathbf c}))$ correspond to 
bordism classes of maps $f:(\Sigma,\partial \Sigma)\to (K'(G,1),L_{\mathbf c})$. 

The maps 
$f$ and $f': (\Sigma',\partial \Sigma')\to  (K'(G,1),L_{\mathbf c})$) are bordant if they extend to a 
3-manifold. This means that that there exists a 3-manifold $M$ whose boundary splits as 
$\partial M=\partial_+M \cup \partial_0 M\cup \partial_-M$, where $\partial_+M=\Sigma$ and $\partial_-M=\Sigma'$,  
and a map  $F:(M, \partial_0M)\to (K'(G,1),L_{\mathbf c})$, which restricts to $\partial_{\pm}M$ to $f$ and $f'$. 
We can assume that $\partial_0 M$ is a trivial cobordism and moreover 
$F:\partial_0 M\to L_{\mathbf c}$  is a product projection. 

Take then a Heegaard surface $(\Sigma'',\partial \Sigma'')$  of the triad $(M,\partial_+M,\partial M_-)$, as 
in \cite{CG}. This means that $\partial \Sigma''$ is the union of core circles of $\partial_0M$ and  $\Sigma''$ 
decomposes $M$ into two compression bodies $H$ and $H'$. We can obtain such Heegaard decompositions by 
extending smoothly to $M$ a function which takes constant values on $\partial_{\pm}M$ and perturb it away from the boundary 
to become Morse. Assuming that $\Sigma$ and $\Sigma'$ are connected we obtain a Heegaard surface after attaching index one handles away from $\partial_+M$. 

It follows that the map induced by $F_*$ on the image of $\pi_1(\Sigma'')$ within $\pi_1(M)$ 
is a common stabilization of the homomorphisms $f$ and $f'$, up the the action of the gluing 
homeomorphism of the two compression bodies $H$ and $H'$.  
\end{proof}

We now prove: 
\begin{lemma}\label{surjective}
Let $\tilde{G}$ be a finitely generated group, $\mathbf{\tilde{c}}\in \tilde G^p$ and $a\in H_2(\tilde{G}, \mathbf{\tilde{c}})$. Then there is a 
compact surface $\Sigma$ and a surjective homomorphism $\phi:\pi_1(\Sigma)\to \tilde{G}$
such that $\phi_*([\Sigma,\partial \Sigma])=a$ and $(f(\gamma_i))_{i=1,\cdots,p}=\mathbf{\tilde{c}}\in \tilde{G}^p$.  
\end{lemma} 
\begin{proof}
Let first $a=0$ and $\mathbf{\tilde{c}}$ empty. For large enough $n$ there exists a surjective homomorphism 
$\psi:\mathbb F_n\to \tilde{G}$. Consider then $\phi_0=\psi\circ i_*$, where 
the homomorphism $i_*:\pi_1(\Sigma_{n})\to  \pi_1(H_{n})=\mathbb F_n$ is induced by 
the inclusion $i$ of $\Sigma_n$ into the boundary of the genus $n$ handlebody $H_n$.
Note that $i_*$ is a surjection. Then ${\phi_0}_*([\Sigma_n])=0$, as $\phi_0$ factors through a free group.

Let now $a\in H_2(\tilde{G}, \mathbf{\tilde{c}})$ be arbitrary.  Pick up a homomorphism 
$\psi_a:\pi_1(\Sigma_{m,p})\to \tilde{G}$ realizing the class $a$, so that 
$(\psi_a(\gamma_i))_{i=1,\cdots,p}=\mathbf{\tilde{c}}\in \tilde{G}^p$. 
By crushing the genus $n$ separating loop on $\Sigma_{n+m,p}$ to a point we obtain a surjective 
homomorphism $\pi:\pi_1(\Sigma_{n+m,p})\to \pi_1(\Sigma_n)*\pi_1(\Sigma_{m,p})$ onto the fundamental group of the join $\Sigma_n \vee \Sigma_{m,p}$. 
Consider further the homomorphism 
$\psi_0*\psi_a: \pi_1(\Sigma_n)*\pi_1(\Sigma_{m,p})\to \tilde{G}$. 

Then  the composition   $\phi_a: \pi_1(\Sigma_{n+m,p})\to \tilde G$, $\phi_a=(\psi_0*\psi_a) \circ \pi$ is surjective. 
Further, by Mayer-Vietoris we have 
\[ H_2(\Sigma_n \vee \Sigma_{m,p}, \partial \Sigma_{m,p}))=H_2(\Sigma_n)\oplus H_2(\Sigma_{m,p}, \partial \Sigma_{m,p})\]
and $\pi_*([\Sigma_{n+m,p}, \partial \Sigma_{n+m,p}])=([\Sigma_n], [\Sigma_{m,p},\partial \Sigma_{m,p}])$. 
This implies that ${\phi_a}_*([\Sigma_{m+n,p}])=a$.  Thus $\phi_a$ satisfies our requirements. 
\end{proof}

\begin{proposition}\label{stablyliftable}
Consider $\mathbf{\tilde{c}}\in \tilde{G}^p$. The surjective homomorphism 
$f:\pi_1(\Sigma_g\setminus B)\to G$ stably lifts to a (surjective) homomorphism 
$\varphi:\pi_1(\Sigma_h\setminus B)\to \tilde{G}$ 
satisfying the constraints $\varphi(\gamma_i)=\tilde c_i$, for $1\leq i\leq p$, 
if and only if there exists a class in $a\in H_2(\tilde{G},\mathbf{\tilde{c}})$ such that 
\[ p_*(a)= sc(f)=f_*([\Sigma,\partial \Sigma])\in H_2(G, p(\mathbf{\tilde{c}})).\]
\end{proposition}  
\begin{proof} 
If $\phi$ is the map provided by Lemma \ref{surjective} above, 
then $p\circ \phi$ and $f$ are two surjective 
homomorphisms having the same Schur class. By the previous Proposition \ref{Livingston} 
they have equivalent stabilizations. This shows that $f$ is stably equivalent with a liftable 
homomorphism and hence stably liftable.  
\end{proof}

\subsection{Finite target groups}
In case when the group $G$ is {\em finite} there is an improvement of 
the stable equivalence of surface group epimorphisms, following the Dunfield-Thurston 
Theorem (\cite{DT}, Thm.6.20) and we can state:  

\begin{proposition}\label{Dunfield-Thurston}
Let $G$ be a finite group. There exists $g(G,m)$ with the property that
any two surjective homomorphisms 
$f, f':\pi_1(\Sigma_g\setminus B) \to G$ with $g\geq g(G,m)$ and  $f(\gamma_i)=f'(\gamma_i)=c_i\in G$ 
having the  same class in $H_2(G, \mathbf c)$ are equivalent under the action of ${\rm SAut}^+(\pi_1(\Sigma_g\setminus B))$. 
\end{proposition}
\begin{proof}
The proof from (\cite{DT} Thm. 6.20 and 6.23) extends without major modifications. In fact if $g > |G|$ 
any homomorphism $f:\pi_1(\Sigma_g\setminus B)\to G$ is a stabilization and this produces 
a surjective homomorphism induced by stabilization 
\[M_G(\Sigma_g,B,{\mathbf c}) \to  M_G(\Sigma_{g+1},B,{\mathbf c})\]
It follows that the cardinal of the orbits set is eventually constant. On the other hand, by 
Proposition \ref{Livingston}, the orbits set eventually injects into $H_2(G,{\mathbf c})$.  
\end{proof}

\begin{proposition}\label{liftable} 
Let $G$ be a finite group $G$, $p:\tilde{G}\to G$ be a surjective homomorphism 
and  $\mathbf c\in G^m$ such that  $\prod_i c_i\in [G,G]$. 
\begin{enumerate}
\item There exist lifts 
$\mathbf{\tilde{c}}\in \tilde G^m$ such that $p(\mathbf{\tilde{c}})=\mathbf c$ and 
$\prod_i \tilde{c_i}\in [\tilde{G},\tilde{G}]$.
\item Given a lift  $\mathbf{\tilde{c}}$ as in the previous item,    
there exists some $g(G,m, \tilde{G}, \mathbf{\tilde{c}})$ such that every 
surjective homomorphism 
$f:\pi_1(\Sigma_g\setminus B)\to G$  with $f(\gamma_i)=c_i\in G$,  $g\geq g(G,m, \tilde{G}, \mathbf{\tilde{c}})$, for which there exists some class in $a\in H_2(\tilde{G}; \mathbf{\tilde{c}})$ satisfying  
\[ p_*(a)=sc(f)\in H_2(G, {\mathbf c})\]
lifts to  $\varphi:\pi_1(\Sigma_g\setminus B)\to \tilde{G}$ 
with the constraints $(\varphi(\gamma_i))_{i=1,\cdots,m}=\mathbf{\tilde{c}}\in \tilde{G}^m$. 
\end{enumerate}
\end{proposition}
\begin{proof}
Let $K$ denote the kernel of the surjection $\tilde{G}\to G$. 
Choose any lift $\mathbf{\tilde{c}}\in \tilde G^m$. Then  
the product of its components differs from a product of commutators by some element $k\in K$.  
We can correct this by replacing the lift $\tilde{c_1}$ by  $k^{-1}\tilde{c_1}$. This proves the first claim. 

Consider the finite set of all pairs $({\mathbf c},s)$, where $\mathbf c\in G^m$ is an $m$-tuple which admits a lift $\mathbf{\tilde{c}}\in\tilde{G}^m$  and some class 
$a_{\mathbf c}\in  H_2(\tilde{G}; \mathbf{\tilde{c}})$ projecting onto 
the primitive class $s\in H_2(G, {\mathbf c})$. 

By Lemma \ref{surjective}
there exists a punctured surface $\Sigma_{k(a_{\mathbf c})}\setminus B$ and a continuous map defined on the compact surface with boundary $(\Sigma',\partial \Sigma')$ 
which compactifies it, say   $(\Sigma',\partial \Sigma')\to (K'(\tilde{G},1), L_{\mathbf c})$,  
which induces a surjective homomorphism 
$\phi: \pi_1(\Sigma_{k(a_{\mathbf c})}\setminus B)\to \tilde{G}$
such that $\phi_*([\Sigma',\partial \Sigma'])=a_{\mathbf c}$. 
Let then $g_0=g_0(G,m, \tilde{G}, \mathbf{\tilde{c}})$ be the maximum of all $k(a_{\mathbf c})$.

If $g\geq g_0$ we stabilize $\phi$ to be defined on $\Sigma_g\setminus B$. 
Let then $\varphi=p\circ \phi$.   Then $\varphi$ is a surjective homomorphism onto $G$ and 
\[ \varphi_*([\Sigma', \partial \Sigma'])=f_*([\Sigma,\partial \Sigma])\in 
H_2(G, p({\mathbf{\tilde{c}}}))\]
Now, from Proposition \ref{Dunfield-Thurston}.   
there exists some $g(G)$ such that for $g\geq g(G)$ any two surjective 
homomorphisms $\varphi$ and $f$ as above are  
equivalent up to the action of  ${\rm SAut}^+(\pi_1(\Sigma_g\setminus B))$. 
We can take $g(G,m, \tilde{G}, \mathbf{\tilde{c}})=\max(g(G), g_0(G,m, \tilde{G}, \mathbf{\tilde{c}}))$. 
The action of ${\rm SAut}^+(\pi_1(\Sigma_g\setminus B))$ preserves 
the set of homomorphisms  $\pi_1(\Sigma_{g}\setminus B)\to G$ which admit a lift to $\tilde{G}$ 
with the given constraints, thereby proving our claim. 
\end{proof}

A directly related question is whether a surjective 
homomorphism $f:\pi_1(\Sigma_g)\to G$  with vanishing Schur class factors through a free group $\mathbb F$, 
in which case of course it can be lifted along any epimorphism $p:\tilde{G}\to G$, for any group $\tilde{G}$.
If this happens, the homomorphism $f$ will be called {\em free}, or {\em elementary}.  
As can be inferred from the previous results we have: 
\begin{proposition}\label{free}
Let $G$ be a finite group. Then there is some $g(G)$ such that for any $g\geq g(G)$ every 
surjective homomorphism $f:\pi_1(\Sigma_g)\to G$ with $f_*([\Sigma_g])=0\in H_2(G)$ is elementary.  
\end{proposition}
\begin{proof}
We can take $\widetilde{G}$ to be a fixed free group surjecting onto $G$ and use Proposition \ref{liftable}. 
\end{proof}

We now need some preliminary results concerning braid groups. 

\begin{lemma}\label{H_2surjective}
Let $G\subseteq S_n$ be a finite group and $\tilde G\subset B_n$ be the preimage 
of $G$ by the projection homomorphism $p:B_n\to S_n$. 
Then the map $p_*:H_2(\tilde{G})\to H_2(G)$ is surjective. 
\end{lemma}
\begin{proof}
The kernel of $p$ is the pure braid group $P_n$ on $n$ strands. The five term exact sequence in homology reads: 
\[H_2(\tilde G)\to H_2(G)\to H_1(P_n)_G\to H_1(\tilde{G})\to H_1(G)\to 0\]
On one hand $H_2(G)$ is a torsion group, as $G$ is finite. 
Furthermore $H_1(P_n)$ is the free abelian group generated by the set $S(n)$ of classes
$A_{ij}$, $1\leq i < j \leq n$ and the action of $S_n$ is 
\[ \sigma\cdot A_{ij}=A_{\min(\sigma(i),\sigma(j)), \max(\sigma(i),\sigma(j))}.\]
By (\cite{Brown}, II.2.ex.1) the module of co-invariants 
$H_1(P_n)_G=\Z S(n)_G$ is isomorphic to the free abelian group 
$\Z[S(n)/G]$. In particular any homomorphisms 
$H_2(G)\to H_1(P_n)_G$ must be trivial. Then the exact sequence above implies the claim. 
\end{proof}

\subsection{Lifting permutations to completely splittable braids}
\begin{lemma}\label{commlifts}
Every $m$-tuple ${\mathbf\sigma}\in S_n^m$ satisfying $\prod_i \sigma_i\in [S_n,S_n]$ 
has a lift 
$\mathbf{\tilde{\sigma}}\in B_n^m$ with the properties:
\begin{enumerate}
\item   $\mathbf{\tilde{\sigma}}\in \mathcal A_n^m\subset B_n^m$; 
\item $\prod_i \tilde{\sigma_i}\in [B_n,B_n]$;
\item Suppose that $\{r_1,r_2,\ldots, r_\nu, t_1,t_2,\ldots, t_\nu\}\subseteq \{1,2,\ldots,m\}$ 
is a subset of the set of indices with the property $\sigma_{r_s}=\sigma_{t_s}^{-1}$, 
for $1\leq s\leq \nu$.  
Then we can choose $\tilde\sigma_i$ such that additionally:  
\[\tilde\sigma_{r_s}=\tilde\sigma_{t_s}^{-1},\; {\rm for}\;  1 \leq s\leq \nu\] 
\end{enumerate} 
\end{lemma}
\begin{proof}
Let $b_i$, $1\leq i\leq n-1$, denote the standard generators of the braid group $B_n$. 
Recall that the exponent sum $e:B_n\to \Z$ is the unique homomorphism taking values 
$e(b_i)=1$ on the standard generators. Set also $\tau_j$ for the transposition $(j,j+1)$ of $S_n$.  

We first show that any permutation $\sigma$ has a lift $\tilde\sigma\in \mathcal A_n$, such that 
$e(\tilde\sigma)=\mu$, where $\mu$ is any prescribed element of 
$\{-1,1\}$, when $\sigma$ is odd and $\mu=0$, otherwise. This is obvious for $n=2$. We proceed by induction when $n>2$, by assuming the claim for $n-1$. Any $\sigma\in S_n$ which does not belong to $S_{n-1}$ can be written as $\sigma=\alpha \tau_{n-1}\beta$, where $\alpha,\beta\in S_{n-1}$.
Pick-up some $\mu\in\{-1,0,1\}$ which is compatible with the parity of $\sigma$, as asked above.  
Choose an arbitrary lift $\tilde\beta\in B_n$. By the induction hypothesis we can find a lift 
$\widetilde{\beta\alpha}\in \mathcal A_{n-1}$  with prescribed  exponent sum $\nu\in\{-1,0,1\}$, depending on the parity of the permutation $\beta\alpha\in S_n$. 
We then define the lift:  
\[ \tilde\sigma= \tilde\beta^{-1} \cdot \widetilde{\beta\alpha} \cdot \tau_{n-1}^{\delta}\cdot \tilde\beta\]
where we set:  
\[ \delta=\left\{\begin{array}{cl}
\mu, & {\rm if} \; \mu\in\{-1,1\}; \\
-\nu, &{\rm if } \; \mu=0.
\end{array}
\right.
\]
Now $\tilde\sigma$ is conjugate to $\widetilde{\beta\alpha} \cdot \tau_{n-1}^{\delta}$ which 
is a stabilization of $\widetilde{\beta\alpha}$ and hence it has the same link closure as the latter. Therefore 
$\tilde\sigma\in \mathcal A_n$, proving the induction step and hence the claim. 
Moreover, we can take $\widetilde{\sigma^{-1}}=  \tilde\beta^{-1} \cdot \tau_{n-1}^{-\delta}\cdot (\widetilde{\beta\alpha})^{-1} \cdot \tilde\beta$, as a lift of $\sigma^{-1}$, which still belongs to $\mathcal A_n$. 

We can actually find explicit lifts $\tilde\sigma$, as follows. 
Recall that the half-braid  $b_{i,j}$ is defined as: 
\[ b_{i,j}=b_i b_{i+1}\cdots b_{j-2}b_{j-1} b_{j-2}^{-1}\cdots b_{i+1}^{-1}b_i^{-1}, {\rm for }\;  i<j \]
\[ b_{i,j}=b_{j,i}^{-1}, {\rm for }\;  i>j \]

To every permutation cycle $c=(i_1,i_2,\ldots,i_k)\in S_n$ and map 
$\varepsilon:\{i_1,i_2,\ldots,i_{k-1}\}\to \{\pm 1\}$, to be called cycle signature, we associate 
a signed mikado braid, as follows: 
\[ \beta(c, \varepsilon)= b_{i_1,i_2}^{\varepsilon(i_1)} b_{i_2,i_3}^{\varepsilon(i_2)}
\cdots b_{i_{k-1},i_k}^{\varepsilon(i_{k-1})}\in B_n\]

Now, every permutation $\sigma\in S_n$ is the product of 
disjoint cycles, say $\sigma=c_1c_2\cdots c_s$.  
Pick-up a cycle signature  $\varepsilon_i$ for each cycle $c_i$. We then set:  
\[ \beta(\sigma, (\varepsilon_i))=\beta(c_1,\varepsilon_1)\beta(c_2,\varepsilon_2)\cdots \beta(c_s,\varepsilon_s)\]
Observe that $\beta(c,\varepsilon)$ and $\beta(c',\varepsilon')$ commute with each other if the cycles $c$ and $c'$ are disjoint. This implies that:  
\[  p(\beta(\sigma, (\varepsilon_i))=\sigma\]

Note that the closure of $\beta(c, \epsilon)$ is a trivial link, for any cycle $c$. 
In fact, we can assume up to a conjugacy, that the cycle $c$ has the form 
$(1,2,\ldots,k)$, so that up to a conjugacy in $B_n$ we have: 
\[ \beta(c, \epsilon)= b_{1}^{\varepsilon(1)} b_{2}^{\varepsilon(2)}
\cdots b_{k-1}^{\varepsilon(k-1)}\in B_n\]
Now we see that this is an iterated stabilization of a trivial braid and hence its closure is a trivial link, regardless of the cycle signature.   
Moreover, the closure of a product of such braids $\beta(c_i, \epsilon_i)$ associated to disjoint cycles $c_i$ is split, each cycle providing a single component of the link.
Thus  $\beta(\sigma, (\varepsilon_i))$ is a completely split unlink.

We can always choose the  cycle signature $\varepsilon$ of a given cycle $c$ such that 
$e(\beta(c, \epsilon))=0$, when $c$ has odd length and $e(\beta(c, \epsilon))=1$,   
otherwise. By changing the cycle signature above to its negative $-\varepsilon$  we can also find 
a cycle signature such that $e(\beta(c, -\epsilon))=-1$, if the length of $c$ is even. 
If $\sigma$ is the product of disjoint cycles $c_i$ we can find  
some cycle signatures $\varepsilon_i$ such that $e(\beta(\sigma, (\varepsilon_i))\in\{-1,0,1\}$, 
by summing up factors with  $e(\beta(c_i, \varepsilon_i)\in\{-1,0,1\}$.

Let now $\sigma_i\in S_n$ be a collection of permutations such that 
$\prod_i \sigma_i\in [S_n,S_n]$. 
We then set 
$\tilde\sigma_i=\beta(\sigma_i, (\varepsilon_{ij}))$ for suitable cycle signature maps $\varepsilon_{ij}$, such that 
$e(\beta(\sigma_i, (\varepsilon_{ij}))\in\{-1,0,1\}$ and also 
$e(\prod_i \beta(\sigma_i, (\varepsilon_{ij}))\in\{-1,0,1\}$.  

If $\sigma_{r_s}=\sigma_{t_s}^{-1}$, then the decompositions into cycles correspond 
bijectively to each other. For each cycle $c_{r_sj}$ of length $k_j$ arising in the 
decomposition of $\sigma_{r_s}$ the cycle $c_{t_sj}=c_{r_sj}^{-1}$ appears in the decomposition of 
$\sigma_{t_s}^{-1}$. We then set: 
\[\varepsilon_{t_sj}(q)=\varepsilon_{r_sj}(k_j-q)\]
This implies that  
\[\beta(c_{t_sj}, (\varepsilon_{t_sj}))= \beta(c_{r_sj}, (\varepsilon_{r_sj}))^{-1}\]
and hence:
\[\tilde\sigma_{r_s}=\beta(\sigma_{t_s}, (\varepsilon_{t_sj}))= \beta(\sigma_{r_s}, (\varepsilon_{r_sj}))^{-1}=\tilde\sigma_{t_s}^{-1}\]

Since  $\prod_i \sigma_i$ is an  even permutation, 
$e(\prod_i\tilde\sigma_i)$ must be even and 
hence it must vanish, since it belongs to $\{-1,0,1\}$.  
This implies that $\prod_i\tilde\sigma_i\in [B_n,B_n]$, proving the claim. 
\end{proof}
 
\begin{remark}
We have a large freedom in the choice of lifts $\tilde\sigma\in \mathcal A_n$. 
Indeed any braid conjugate to $\beta(\sigma,(\varepsilon_i))$ is in $\mathcal A_n$ 
and the corresponding product still belongs to the commutator subgroup of $B_n$. 
\end{remark}

\subsection{End of proof of Theorem \ref{stable}}
We further have the following result, which  along with Corollary \ref{corliftable} proves Theorem \ref{stable}.   
Several arguments of the proof have essentially been discussed in \cite{AMS}.

\begin{theorem}\label{finitelift}
There exists some 
$h_{n,m}$ such that if $g\geq h_{n,m}$, then every homomorphism  
$f:\pi_1(\Sigma_{g}\setminus B)\to S_n$,  
admits a lift $\varphi: \pi_1(\Sigma_{g}\setminus B)\to B_n$ 
satisfying $\varphi(\gamma_i)\in \mathcal A_n$. 
\end{theorem}
The case $n=3$ and $B=\emptyset$ was solved in \cite{HMM}.
\begin{proof}
The group $S_n$ is generated by two elements $a$ and $b$, for instance a $n$-cycle and a transposition. 
Set $B'$ for the result of adding $4$ more points $p_{m+1},p_{m+2}, p_{m+3}, p_{m+4}$ 
to $B$. There is a natural surjection  $\pi_1(\Sigma_{g}\setminus B')\to \pi_1(\Sigma_{g}\setminus B)$ which corresponds to removing the extra punctures. 
Define the lift $f':\pi_1(\Sigma_{g}\setminus B')\to S_n$ of $f$ by asking that 
the monodromy  $\sigma_{m+i}$ around a loop encircling once counterclockwise $p_{m+i}$, for $1\leq i\leq 4$ be 
$a, a^{-1}, b$ and $b^{-1}$, respectively. 
By construction $f'$ is a surjective homomorphism onto $S_n$.

By Lemma \ref{commlifts} we can choose for 
each $(m+4)$-tuple ${\mathbf\sigma}\in S_n^{m+4}$ some 
lift $\mathbf{\tilde{\sigma}}\in \mathcal A_n^{m+4}\subset B_n^{m+4}$ with 
$\prod_i \tilde{c_i}\in [B_n,B_n]$ and $\tilde\sigma_{m+1}\tilde\sigma_{m+2}=1$, 
$\tilde\sigma_{m+3}\tilde\sigma_{m+4}=1$. 
Let $h_{n,m}$ be the maximum of $g(S_n,m+4, B_n, \mathbf{\tilde{\sigma}})$, over all  
$\mathbf\sigma\in S_n^{m+4}$. 

We  claim that we can lift $f'$ to $B_n$ with the constraints $\mathbf{\tilde{\sigma}}\in B_n^{m+4}$. 
By Proposition \ref{liftable},   it suffices to prove that the homomorphism 
$p_*:H_2(B_n, \mathbf{\tilde{c}})\to  H_2(S_n, \mathbf c)$ 
surjects onto the primitive classes. We have a commutative diagram: 
\[ \begin{array}{ccccccccccc}
0 &\to& H_2(B_n)&\to& H_2(B_n,{\mathbf {\tilde{c}}})&\to& \Z^{m+4}&\to&H_1(B_n)& \to 0\\
  &    & \downarrow & & \downarrow & & \downarrow & &\downarrow & \\    
0 &\to& H_2(S_n)&\to& H_2(S_n,{\mathbf c})&\to& \Z^{m+4}&\to&H_1(S_n)& \to 0\\
\end{array}
\]
where the rightmost vertical arrow is the homomorphism induced by the projection 
$ H_1(B_n)\to H_1(S_n)$, which is surjective. Note that 
the third vertical arrow is $H_1(L_{\mathbf{\tilde{c}}})\to H_1(L_{\mathbf{{c}}})$, 
which is an isomorphism. Then the five-lemma reduces the surjectivity 
claim to the surjectivity of $p_*:H_2(B_n)\to H_2(S_n)$, which was proved in Lemma \ref{H_2surjective}.

Therefore $f'$ lifts to a homomorphism $\varphi':\pi_1(\Sigma_{g}\setminus B')\to B_n$. 
By removing from $\Sigma_g\setminus B$ two disks containing  the pairs $p_{m+1},p_{m+2}$ and 
$p_{m+3}, p_{m+4}$ respectively, we obtain a surface with boundary, whose fundamental 
group injects into  $\pi_1(\Sigma_{g}\setminus B')$. The homomorphism $\varphi'$ takes trivial 
values on the loops around each of the two holes. Therefore 
$\varphi'$  induces a homomorphism of the fundamental group of the surface 
obtained by capping off the boundary components by disks, 
namely a homomorphism $\varphi:\pi_1(\Sigma_g\setminus B)\to B_n$. This is the desired lift for $f$.  
\end{proof}

\section{Thickness of  elementary surface group homomorphisms}
\subsection{Elementary homomorphisms and 3-manifolds}
For the sake of simplicity, we stick in this section to the unramified case $B=\emptyset$. 
Let $f:\pi_1(\Sigma_g)\to G$ be a surjective homomorphism. Assume that 
$f_*([\Sigma_g])=0\in H_2(G)$. 
Then there exists some 3-manifold $M^3$ with boundary $\Sigma_g$ such that 
$f$ extends to $F:\pi_1(M^3)\to G$ (see the proof of Propositions \ref{Livingston} and \ref{prop:genus}). 

\begin{definition}
The {\em thickness} $t(f)$ of the surjective homomorphism $f:\pi_1(\Sigma_g)\to G$ with $sc(f)=0$ is the smallest value of $n$ for which 
there exists a  3-manifold $M^3$ with boundary $\Sigma_g$ and Heegaard genus $g+n$ such that 
$f$ extends to $F:\pi_1(M^3)\to G$. 
\end{definition} 
Other meaningful version might be the rank of the homology or the rank of $\pi_1(M^3)$, 
the hyperbolic volume (when $g=1$) of $M^3$ or any other complexity function on 3-manifolds.

This situation generalizes the case of the commutator width of  
elements in $[G,G]$. On the other hand it  is an analog of Thurston's norm on the 
homology $H_2(M)$ of a 3-manifold.

\begin{proposition}\label{prop:genus}
Let $f:\pi_1(\Sigma_g)\to G$ be a  homomorphism satisfying 
$f_*([\Sigma_g])=0\in H_2(G)$.  
The minimal genus $h$ for which there exists a stabilization  
$f':\pi_1(\Sigma_h)\to G$ which is elementary equals the  
minimal Heegaard genus a 3-manifold $M^3$ with boundary $\Sigma_g$ such that 
$f$ extends to a homomorphism $\pi_1(M^3)\to G$.  
\end{proposition}
\begin{proof}
The arguments come from Livingston's proof (\cite{Liv,DT}) of 
the stable equivalence of homomorphisms. 
Observe that there is a map $F:\Sigma_g\to B G$ inducing 
$f$ at the fundamental group level. Our assumptions and Thom's solution to the Steenrod realization problem implies that there is some 
3-manifold $M^3$ with boundary $\Sigma_g$ such that $F$ extends to a map still denoted by the same letter 
$F:M^3\to B G$. It follows that $f$ factors as the composition 
\[ \pi_1(\Sigma_g)\to \pi_1(M^3)\stackrel{F_*}{\to} G\]
where $\pi_1(\Sigma_g)\to \pi_1(M^3)$ is the homomorphism induced by the inclusion.

 Let $\Sigma_k$ be a Heegaard surface in $M^3$, bounding a handlebody 
$H_k$ of genus $k$ on one side and a compression body $H_{k,g}$ on the other side.
Recall that a compression body $H_{k,g}$ is a compact orientable irreducible 3-manifold 
obtained from $\Sigma_k\times[0,1]$ by adding 2-handles with disjoint attaching curves, so that 
$\pi_1(\Sigma_k)\to \pi_1(H_{k,g})$ is surjective. Alternatively we can see 
 $H_{k,g}$ as the result of adding to $\Sigma_{g}\times [0,1]$ a number of 1-handles, so that 
 $\pi_1(H_{k,g})=\pi_1(\Sigma_g)*\mathbb F_{k-g}$. 
 We then have $\pi_1(M^3)=\pi_1(H_k)*_{\pi_1(\Sigma_k)}\pi_1(H_{k,g})$ where all homomorphisms 
 are induced by the inclusions. 
 
Observe that $f$ is given by the composition: 
\[  \pi_1(\Sigma_g)\to \pi_1(H_{k,g})\to \pi_1(M^3)\stackrel{F_*}{\to} G\]
where the first two arrows are inclusion induced homomorphisms. Consider now 
the homomorphism $f'$ defined by the composition 
\[ \pi_1(\Sigma_k)\to \pi_1(H_{k,g})\to \pi_1(M^3)\stackrel{F_*}{\to} G\]
where the first two arrows are inclusion induced homomorphism. Since $F_*$ extends $f$, it follows that 
$f'$ is a stabilization of $f$ (see also \cite{DT} section 6.15). On the other hand $f'$ factors through the free group $\pi_1(H_k)$.    
It follows that $h$ is bounded by the Heegaard genus,  $h\leq k$.

Conversely, let $f':\pi_1(\Sigma_h)\to G$ be a stabilization of $f$ which factors 
through a free group $\mathbb F$, namely we can write it as 
$f'=q'\circ \rho$, where $q':\mathbb F\to G$ and $\rho:\pi_1(\Sigma_h)\to \mathbb F$.  

Recall the following lemma due to Zieschang, Stallings and Jaco (see \cite{Zie}, \cite[Lemma 3.2]{Jaco}) in the form presented by Liechti and March\'e (\cite{LM}, Lemma 3.5):
\begin{lemma}\label{LiechtiMarche}
Let $\Sigma_h$ be a surface bounding a handlebody $H_h$ and $\mathbb F$ a free group. 
Then any homomorphism $\rho:\pi_1(\Sigma_h)\to \mathbb F$ factors as 
$q\circ i_* \circ \phi_*$, where $\phi_*$ is an automorphism of $\pi_1(\Sigma_h)$ preserving the 
orientation, $i:\pi_1(\Sigma_h)\to \pi_1(H_h)$ is the inclusion and 
$q:\pi_1(H_h)\to \mathbb F$ is a homomorphism. 
\end{lemma}
Write then $\rho=q\circ i_* \circ \phi_*$ as in Lemma  \ref{LiechtiMarche} and define the manifold 
$M^3=H_h\cup_{\phi} H_{h,g}$, where the gluing homeomorphism $\phi$ induces the automorphism $\phi_*$. 
It then follows that $f'$ factors through $\pi_1(M^3)= \pi_1(H_h)*_{\pi_1(\Sigma_h)}\pi_1(H_{h,g})$. Since $\Sigma_h$ is a Heegaard surface in $M^3$ we derive that 
$k\leq h$.   
\end{proof}
  
\begin{corollary}
There is some $h_n$ such that whenever $g\geq h_n$ and 
$f:\pi_1(\Sigma_g)\to G\subseteq S_n$ is a homomorphism with $f_*([\Sigma_g])=0\in H_2(G)$, then 
$f$ is equivalent to a homomorphism which factors through $\pi_1(H_g)$. 
\end{corollary}
 
\subsection{Expressing thickness algebraically}
The next result aims at formulating an algebraic formula for $t(f)$, similar to Hopf's formula 
for the second homology. Consider a standard presentation of the group $\pi_1(\Sigma_g)$ 
using the generators system $\{a_{i}, b_i\}_{i\in \{1,2,\ldots,g\}}$ of the form: 
\[ \pi_1(\Sigma_g)= \langle \{a_{i}, b_i\}_{i\in \{1,2,\ldots,g\}} | \; 
\prod_{i=1}^g [a_i,b_i]=1\rangle \]
Then one identifies  a homomorphism $f:\pi_1(\Sigma_g)\to G$ with a labeled set 
$S=\{\alpha_{i}=f(a_i), \beta_i=f(b_i)\}_{i\in \{1,2,\ldots,g\}}$ of elements of $G$ satisfying the condition:  
\begin{equation} 
\prod_{i=1}^g [\alpha_i,\beta_i]=1\in G
\end{equation}

Let $G=\mathbb F/R$ be a presentation of the group $G$, where $\mathbb F$ is a free group and $R$ the normal subgroup generated by the relators.  
For every labeled set 
$\tilde S=\{(\tilde\alpha_{i}, \tilde\beta_i)\}_{i\in  \{1,2,\ldots,g\}}$ of  lifts 
of $S$ to $\mathbb F$ we set: 
\begin{equation} 
 ocl(\tilde S) = \min\{n | {\rm there \; exist} \; f_j\in \mathbb F, r_j\in R, j=1,2,\ldots,n \; {\rm with}  \prod_{i=1}^g [\tilde\alpha_i,\tilde\beta_i]=\prod_{j=1}^n [r_j,f_j]\} 
\end{equation}
Note that such $n$ exists. Indeed the  Hopf formula provides us with an isomorphism 
\[ H_2(G)=\frac{[\mathbb F,\mathbb F]\cap R}{[\mathbb F, R]}\]
Under this identification the $f_*([\Sigma_g])$  is represented by the class 
of the element $\prod_{i=1}^g [\tilde\alpha_i,\tilde\beta_i]\in [\mathbb F,\mathbb F]\cap R$. 
As $f_*([\Sigma_g])$ vanishes by our assumptions $\prod_{i=1}^g [\tilde\alpha_i,\tilde\beta_i]\in[\mathbb F, R]$.

Eventually we define: 
\begin{equation} 
 ocl(f)= \min \{ocl(\tilde S) |  \tilde S \; {\rm lifts} \; S \}
\end{equation}

We then have the following:  
\begin{proposition}
The minimal number of stabilizations needed for making $f$  elementary is $ocl(f)-g$.  
\end{proposition}
\begin{proof}
By Proposition \ref{prop:genus} the minimal $h=g+n$ which appears above is the 
minimal Heegaard genus of a manifold $M^3$ with boundary $\Sigma_g$ such that  $f$ extends   
to some homomorphism $F_*:\pi_1(M^3)\to G$. It remains to prove that the smallest Heegaard genus coincides with  $ocl(f)$. 
The proof goes similarly with that given by Liechti-March\'e \cite{LM} for the case of a bordant torus.

Let $\Sigma_h$ a Heegaard surface in $M^3$.  Take a standard system of generators  
of $\pi_1(\Sigma_h)$ of the form $\{a_j,b_j\}_{j=1,\ldots,h}$ such that all $b_j$ bound disks in the handlebody $H_h$. 
By adjoining 2-handles to $\Sigma_h\times[0,1]$ along $b_j\times \{0\}$, for all $j\not\in \{1,2,\ldots,g\}$,  
we obtain a compression body $H_{h,g}$ and we can write $M^3=H_{h,g}\cup_{\phi} \overline{H_h}$, for some gluing homeomorphism $\phi$. 
We have then surjective homomorphisms 
$\pi_1(\Sigma_h)\to \pi_1(H_h)$ and $\pi_1(\Sigma_h)\to \pi_1(H_{h,g})$, while 
$\pi_1(M^3)= \pi_1(H_{h,g}) *_{\pi_1(\Sigma_h)} \pi_1(H_h)$. 

Denote by $\theta_*: \pi_1(\Sigma_h)\to \pi_1(M^3)$ the inclusion induced homomorphism. We observed in the proof of Proposition \ref{prop:genus} above 
that $F_*\circ \theta_*:\pi_1(\Sigma_h)\to G$ is a stabilization of  $f$. 
Its key property is that
\[ F_*\circ \theta_*(b_i)=1,\; {\rm  if }\;  i\not\in \{1,2,\ldots,g\}\]

The homomorphism $\theta$ factors through the free group $\pi_1(H_h)$. Therefore  
$F_*\circ \theta_*:\pi_1(\Sigma_h)\to G$  lifts to a homomorphism 
$\tilde{f}:\pi_1(\Sigma_h)\to \mathbb F$. 
Consider the images $\tilde\alpha_j=\tilde f \circ i_*(a_j),\tilde b_j=\tilde f \circ i_*(b_j)$ of  the generators above into $\mathbb F$. 
As $\tilde f\circ i_*$ is a homomorphism, we have the relation:  
\begin{equation}\label{fundamental} \prod_{i=1}^g  [\tilde\alpha_i,\tilde\beta_i] = 
\prod_{i=g+1}^{h}  [\tilde\beta_i, \tilde\alpha_i]\end{equation}
As $\tilde\beta_i\in R$, we derive that 
\begin{equation}
 ocl(f)+ g\leq h
\end{equation}

Conversely, if we have elements  $\tilde\alpha_j,  \tilde\beta_j$ satisfying 
equation (\ref{fundamental}), then we can define a  
homomorphism $\tilde f:\pi_1(\Sigma_h)\to \mathbb F$, by 
$\tilde f(\alpha_j)= \tilde\alpha_j, \tilde f(\beta_j)=\tilde \beta_j$, $1\leq j\leq h$. 
By Lemma \ref{LiechtiMarche} such a homomorphism factors through $\pi_1(H_h)$, namely  is a composition 
 \[ \tilde f=q\circ i_*\circ \phi_*^{-1}\]
where $i_*$ is as above,  $\phi_*$ is an automorphism of $\pi_1(\Sigma_h)$ and $q:\pi_1(H_h)\to \mathbb F$ is some homomorphism. 
 
Let $f': \pi_1(\Sigma_h)\to G$ be the composition of the projection $\mathbb F\to G$ with $\tilde f$. 
Then  $f'\circ \phi_*(b_i)=1$, for $i=1,2,\ldots,h$.  
On the other hand, as $\tilde\beta_j\in R$, for $j>g$, $f'$ factors also through $\pi_1(H_{h,g})$. 
This implies that $f'$ extends to a homomorphism $F_*:\pi_1(M^3)\to G$, where 
$M^3=H_{h,g}\cup_{\phi} \overline{H_h}$. Eventually, note that the restriction of 
$F_*$ to the image of $\pi_1(\Sigma_g)$ within $\pi_1(M^3)$ is $f$. We constructed an extension of $f$ to a 3-manifold of Heegaard genus at most $h$ and thus we proved the reverse inequality 
\begin{equation}
 ocl(f)+ g \geq h
\end{equation} 
\end{proof}

\begin{remark} 
Consider two surjective homomorphisms $f_j:\pi_1(\Sigma_{g_j})\to G$.  
We know that $f_j$ are stably equivalent if and only 
\[ {f_1}_*[\Sigma_1]={f_2}_*[\Sigma_2]\in H_2(G)\]
If $S_j=\{\alpha_i,\beta_i\}_{i\in I_j}$ are images of generators of $\pi_1(\Sigma_j)$ by $f_j$ and 
$\tilde S_j=\{\tilde\alpha_i,\tilde\beta_i\}_{i\in I_j}$ are lifts to $\mathbb F$, we can define 
\[ ocl^s(\tilde S_1,\tilde S_2) = \min\{n |  \prod_{i\in I_1} 
[\tilde\alpha_i,\tilde\beta_i]\prod_{j=1}^{n-g_1} [r_j,f_j]=
\prod_{i\in I_2} [\tilde\alpha_i,\tilde\beta_i]\prod_{j=1}^{n-g_2} [r'_j,f'_j], \; 
  f_j,f'_j\in \mathbb F, r_j,r'_j\in R\} \]
Eventually we set: 
\[ ocl^s(f_1,f_2)= \min \{ocl(\tilde S_1,\tilde S_2) |  \tilde S_j \; {\rm lifts} \; S_j \}\]
If $f_1$ and $f_2$ are stably equivalent, then $ocl^s(f_1,f_2)$ equals the minimal genus
of a Heegaard splitting separating the boundaries of a 3-manifold to which $f_j$ extend and also 
the minimal number of stabilizations yielding equivalent representations in $G$. The proof is identical. 
\end{remark}

\begin{remark}
The branched surface case of homomorphisms $f:\pi_1(\Sigma_g\setminus B)\to G$ 
with prescribed images of peripheral loops follows directly from the closed 
surface treated above, without essential modifications. 
\end{remark}

\section{Nontrivial thickness and proof of Theorem \ref{simplequotients}}
\subsection{Finite simple non-abelian characteristic quotients}
In \cite{FL} the authors proved that the obvious extension of Wiegold's conjecture to  
surface groups does not hold.

Let $\Sigma_g^1$ denote the once punctured closed orientable surface of genus $g$.  
TQFTs provide the so-called quantum representations of punctured mapping class groups  
\[\rho_p:\Gamma(\Sigma_g^1)\to P\mathbb G_p,\] 
for prime $p\equiv 3 \; ({\rm mod} \; 4)$,  into the integral points of a projective pseudo-unitary group $P\mathbb G_p$ defined over $\mathbb Q$.  It is proved in \cite{FL} that  for large enough $p$ (or $p<100$)
the restriction to $\rho_p(\pi_1(\Sigma_g))$ is a Zariski dense subgroup of $P\mathbb G_p$. 
By the Nori-Weisfeiler strong approximation Theorem we obtain (see \cite[Theorem 1.4]{FL}):
\begin{theorem}\label{FL18}
For large prime $p\equiv 3 \; ({\rm mod} \; 4)$ and large enough primes $q$ the reduction mod $q$ of 
the quantum representation $\rho_p$ exists and has the following properties: 
\begin{enumerate}
\item  its restriction to $\pi_1(\Sigma_g)$ is a surjective homomorphism $f_{p,q}:\pi_1(\Sigma_g)\to P\mathbb G_p(\mathbf F_q)$ onto group of  points of  $P\mathbb G_p$ over the finite field $\mathbf F_q$;
\item the finite groups $P\mathbb G_p(\mathbf F_q)$ are finite simple 
groups of Lie type; 
\item $\ker f_{p,q}$ is a characteristic subgroup of $\pi_1(\Sigma_g)$. 
\end{enumerate}
\end{theorem}
\begin{remark}
For all but finitely many $q$ the finite groups $P\mathbb G_p(\mathbf F_q)$ are isomorphic to either $PSL(N_{g,p}, \mathbf F_q)$ or to projective unitary groups $PU(\mathbf F_{q^2})$. 
If $q-1$ is coprime with $N_{g,p}$, then   $PSL(N_{g,p}, \mathbf F_q)$ has vanishing Schur multiplier 
$H_2(PSL(N_{g,p}, \mathbf F_q))=0$. 
\end{remark} 

We now first show that such quotient homomorphism should be non-elementary:

\begin{proposition}\label{nonelementary}
If  $f:\pi_1(\Sigma_g)\to G$ is a surjective homomorphism onto a characteristic finite  
non-trivial quotient $G$, then  $f$ is not elementary, namely its thickness is positive. 
\end{proposition}
\begin{proof}
From Lemma \ref{LiechtiMarche} $f$ is elementary iff there exists some automorphism $\phi$ such that 
$f=h\circ i_*\circ \phi^{-1}$, where  $i_*:\pi_1(\Sigma_g)\to \pi_1(H_g)=\mathbb F_g$ is the inclusion induced  homomorphism and $h:\mathbb F_g\to G$. 
If $\alpha$ is an oriented simple closed curve on $\Sigma_g$ let $\overline{\alpha}$ denote the conjugacy class of 
$\alpha$ in $\pi_1(\Sigma_g)$. Let $\alpha$ be a non-separating simple closed curve on $\Sigma_g$ which bounds a properly embedded disk in $H_g$. Then $i_*(\overline{\alpha})=1$, so that 
$\phi(\overline{\alpha})\in \ker f$.  

Since $\ker(f)$ is a characteristic subgroup of $\pi_1(\Sigma_g)$, 
we also have  $\psi(\overline{\alpha})\subset \ker(f)$ for any automorphism 
$\psi\in\Aut(\pi_1(\Sigma_g))$. However, any  non-separating 
simple closed curve $\gamma$ on $\Sigma_g$ is the image of $\alpha$ by some homeomorphism of the surface. 
In particular, the conjugacy classes of  the simple closed curves from a standard generator system 
of $\pi_1(\Sigma_g)$ are contained into $\ker f$. This implies that $f$ would be constant, which is a contradiction, thereby proving the claim. 
\end{proof}

\subsection{Non-geometric quotients}
We can slightly improve the result above, for the specific case of the homomorphisms $f_{p,q}$ from \cite{FL}.  
\begin{proposition}\label{quantum}
Let $f_{p,q}:\pi_1(\Sigma_g)\to P\mathbb G_p(\mathbf F_q)$ be 
the homomorphisms above, defined by prime $p\equiv 3 \; ({\rm mod} \; 4)$,  and prime $q$ large enough, depending on $p$. 
Then $\ker f_{p,q}$ is non-geometric, namely it contains no 
simple closed curve. 
\end{proposition}
\begin{proof}
The image of the  homotopy class of a based simple closed curve $\gamma$ into $\Gamma(\Sigma_g^1)$ is the 
product of the two commuting left Dehn twists  along the curves ${\gamma^+}$ and ${\gamma^-}$ 
obtained by pushing slightly the curve $\gamma$ off the base point towards the left and right, respectively. 
Therefore the order of $\rho_p(\gamma)$ is equal to the  order of the image of a Dehn twist by the representation $\rho_p$, 
which is known to be $p$ (see \cite{FL}).  

Consider the  projective matrices  $\rho_p(\gamma)^m$, for $1\leq m<p$, where 
$\gamma$ belongs to a finite set of representatives of the set of simple closed  on $\Sigma_g^1$ 
up to the mapping class group action.  Then, for all large enough primes $q$ the reduction mod $q$ of these 
projective matrices are non-trivial. Thus, for every simple closed curve $\gamma$ the elements $f_{p,q}(\gamma)$ have 
order $p$. In particular,  the kernel of $f_{p,q}$ is non-geometric.  
\end{proof}

Theorem \ref{simplequotients} is a consequence Theorem \ref{FL18} along with Propositions \ref{nonelementary} and 
\ref{quantum}.

It is not known what is the largest possible stabilizer of the kernel of an elementary homomorphism. 
The following is relevant: 

\begin{proposition}
Let $\Gamma(H_g^1)$ be the mapping class group of the punctured handlebody $H_g^1$. 
If $f:\pi_1(\Sigma_g)\to G$ is a surjective elementary homomorphism onto a finite quotient whose 
kernel  is invariant by the handlebody subgroup $\Gamma(H_g^1)$, then   
there exists a characteristic finite quotient of $\mathbb F_g$. 
\end{proposition}
\begin{proof}
If $f$ is elementary, then up to composing with an automorphism $\phi$ of $\pi_1(\Sigma_g)$, 
it factors through $i_*:\pi_1(\Sigma_g)\to \pi_1(H_g)=\mathbb F_g$, namely $f=h\circ i_*\circ \phi$, for some 
homomorphism $h:\pi_1(H_g)\to G$.

Recall from \cite{Hensel} that  the mapping class group $\Gamma(H_g^1)$ of the once punctured (or marked) 
handlebody embeds into the mapping class group $\Gamma(\Sigma_g^1)$ of its boundary surface. Moreover, 
Luft (\cite{Luft}) showed that the action in homotopy provides an exact sequence:  
\[ 1\to Tw(H_g^1)\to \Gamma(H_g^1)\to \Aut^+(\mathbb F_{g})\to 1\]
whose kernel is the group of twists, generated by the Dehn twists along meridians of $\Sigma_g^1$ (i.e. curves bounding disks in $H_g^1$).

As $f$ is invariant by $\Gamma(H_g^1)$, the exact sequence 
above shows that the homomorphism $h$ is also invariant by $\Aut^+(\mathbb F_g)$. 
The same argument also work for the full automorphism group. 
\end{proof}

\begin{remark}
Let $P: \pi_1(\Sigma_{g+1})\to \pi_1(\Sigma_g)$ be the map induced by a pinch map 
$P: \Sigma_{g+1}\to \Sigma_g$. Let $\gamma$ be a simple closed curve on $\Sigma_g$ based at the point of $\Sigma_g$ corresponding to the image of the pinched handle. The based homotopy class 
of $\gamma^p$ can be realized by $p$  parallel copies of $\gamma$ which only intersect at the base point. 
Observe that $\gamma^p$ is the image by the pinch map of a simple closed loop $\widehat{\gamma}$ in $\Sigma_{g+1}$. If $f$ is a homomorphism into a group $G$ and $f(\gamma)$ has order $p$ then  
$f\circ P (\widehat{\gamma})=1$ and hence $f\circ P$ has not anymore non-geometric kernel.  
\end{remark}

\begin{remark}
The method used in \cite{FL} also provides epimorphisms $f:\pi_1(\Sigma_g\setminus B)\to G$ 
onto finite simple non-abelian groups $G$, whose kernels  are 
$\Gamma(\Sigma_g\setminus B)$-invariant. 
\end{remark} 
 
\section{Stabilizing cohomology groups}
We now consider approximated lifts of homomorphisms  into $S_n$. 
Let $\gamma_0G=G$, $\gamma_{k+1}G=[\gamma_kG,G]$ denote the lower central series of the group $G$. 
It is well-known that $PB_n$ is residually torsion-free nilpotent, namely 
$\bigcap_{k=0}^\infty \gamma_kPB_n=1$ and $A_{k}=\frac{\gamma_{k-1}PB_n}{\gamma_{k}PB_n}$ are finitely generated torsion-free abelian groups. We denote  $B_n^{(k)}$ the quotient $B_n/\gamma_kPB_n$. 
We then have a series of {\em abelian extensions} 
\[ 1\to  \gamma_kPB_n/\gamma_{k+1}PB_n\to B_n^{(k+1)}\to B_n^{(k)}\to 1\]
The question whether a homomorphism $f_k:\pi_1(\Sigma_g)\to B_n^{(k)}$ admits 
a lift to $f_{k+1}: \pi_1(\Sigma_g)\to B_n^{(k+1)}$ can be reformulated in purely cohomological 
terms. For every $k\geq 1$ there exist examples of homomorphisms $f_k$ which admit no lift.  
Our goal here is to show that the lifting is always possible when $k=0$.

In order to do that, we first show that the pinching map $P$ induces an injection at cohomological level. For the sake of simplicity we only consider the unramified case, but the result works in full generality.  Specifically, let $f:\pi_1(\Sigma_g)\to G$ be a surjective homomorphism and 
$A$ be a finitely generated $G$-module, say by means of a homomorphism 
$\tau:G\to \Aut(A)$. Then $A$ inherits a $\pi_1(\Sigma_g)$-module structure through $\tau\circ f$.
Let now $P:\pi_1(\Sigma_{g+1}) \to \pi_1(\Sigma_g)$ be the pinch map, which is given in convenient 
basis $\{\alpha_i, \beta_i\}_{i=1,\ldots,g+1}$ and $\{a_i, b_i\}_{i=1,\ldots,g}$ by
\[ P(\alpha_i)=a_i, P(\beta_i)=b_i, \, i\leq g, \, P(\alpha_{g+1})=P(\beta_{g+1})=1\]   
Then $f\circ P: \pi_1(\Sigma_{g+1})\to G$ provides a $\pi_1(\Sigma_{g+1})$-module structure on $A$ by means of $\tau\circ f\circ P$. 
Our main result is the following: 

\begin{proposition}\label{injectivestab}
The homomorphism $P^*: H^2(\pi_1(\Sigma_g), A)\to H^2(\pi_1(\Sigma_{g+1}), A)$ is injective. 
\end{proposition}
\begin{proof}
Consider a normalized 2-cocycle $w:\pi_1(\Sigma_g)\times \pi_1(\Sigma_g)\to A$, namely such that $w(x,1)=w(1,x)=0$, 
 whose cohomology class lies in $\ker P^*$. 
Then the image 2-cocycle $P^*w$ is given by 
$P^*w(x,y)=w(P(x),P(y))$. By hypothesis it is exact, namely of the form
\[P^*w(x,y)=\delta \phi(x,y)= x\cdot \phi(y) -\phi(xy) + \phi(x) \]
where 
$\phi:\pi_1(\Sigma_{g+1})\to A$ is a 1-cochain.  
Let $H=\langle \alpha_i, \beta_i, i\in \{1,2,\ldots,g\}\rangle$ 
and $K=\langle \alpha_{g+1},\beta_{g+1}\rangle$ be the subgroups of $\pi_1(\Sigma_{g+1})$ 
generated by the respective elements and note that they are free groups. 
We observe that whenever $x\in H$ and $u,v\in K$ we have: 
\[ \phi(x u)= x\cdot \phi(u)+\phi(x)- w(P(x),P(u))=x\cdot \phi(u) + \phi(x)\]
\[ \phi(ux)=u\cdot \phi(x)+ \phi(u) -w(P(u),P(x))=\phi(x)+\phi(u)\]
\[ \phi(uv)=u\cdot \phi(v)+ \phi(u) -w(P(u),P(v))=\phi(v)+\phi(u)\]
The last equation implies that 
\[ \phi(u^{-1})=-\phi(u), \; {\rm for}\; u\in K,\]
so that 
\[ \phi([\alpha_{g+1},\beta_{g+1}])=0\]
We aim at analyzing the restriction of $\phi$ to $H$. 
Set 
\[ R=\prod_{i=1}^g [\alpha_i,\beta_i]\in H\]
Then 
\[0=\phi(1)=\phi(R [\alpha_{g+1},\beta_{g+1}])=R\cdot \phi(R)= P(R)\cdot \phi(R)=\phi(R)\]
Consider further $x\in H$. By above we have 
\[ \begin{array}{lll}
\phi(xRx^{-1}) &=& xR\cdot \phi(x^{-1})+\phi(x)- P^*w(xR, x^{-1})\\
 &=&xR\cdot (x^{-1}\cdot (P^*w(x,x^{-1}) -\phi(x))+\phi(x)- P^*w(xR, x^{-1})\\
 &=&xRx^{-1}\cdot P^*w(x,x^{-1}) - xRx^{-1}\cdot \phi(x)+\phi(x)- P^*w(xR, x^{-1})\\
 &=&xRx^{-1}\cdot P^*w(x,x^{-1}) - P^*w(xR, x^{-1})\\
\end{array}
\]
the last equality following from the fact that $P(xRx^{-1})=1$ thereby $xRx^{-1}$ is 
acting trivially on $A$. 
Recall that $w$ is a 2-cocycle and hence satisfies 
 \[ x\cdot w(y,z)- w(xy,z)+w(x,yz)-w(x,y)=0\]
Its pullback verifies then 
\[ \begin{array}{lll}
\phi(xRx^{-1}) &=& xRx^{-1}\cdot P^*w(x,x^{-1})- P^*w(xR,x^{-1})=\\
&=& P^*w(xRx^{-1},x)-P^*w(xRx^{-1},1)=
w(P(xRx^{-1}),P(x))=0 \\
\end{array}
\]
It follows that 
\[ \phi(xR^{-1}x)=0\]
Let $L\triangleleft H$  be the normal subgroup generated by $R$ within $H$.
Every element of $L$ can be written as a product of conjugates of $R$ and $R^{-1}$ within $H$. 
If $x,y\in L$ and $\phi(x)=\phi(y)=0$, then 
\[ \phi(xy)=x\cdot \phi(y)+\phi(x)- w(P(x),P(y))=0\]
because $P(x)=P(y)=1$. By induction on the number of conjugates, we derive that 
$\phi|_{L}:L\to A$ is trivial. 

Now, if $x= u y$, where $u\in L$, then 
\[ \phi(x)=\phi(uy)=u\phi(y)+\phi(u)-w(P(u), P(y))=
\phi(y)\]
because $P(u)=1$. It follows that $\phi$ is constant on right cosets of $L$, so that 
$\phi$ induces a well-defined map $\overline{\phi}:H/L\to A$.
Moreover the restriction of the homomorphism 
$P:\pi_1(\Sigma_{g+1})\to \pi_1(\Sigma_g)$ to $H$ induces an isomorphism of $\overline{P}:H/L\to 
\pi_1(\Sigma_g)$. 
  
Observe that the $1$-chain $\overline{\phi}$ satisfies for all $x,y\in H/L$  
\[\delta \overline{\phi}(x,y) =P^*w(\tilde{x},\tilde{y})=w(\overline{P}(x),\overline{P}(y))\]
where $\tilde{x},\tilde{y}$ are lifts in $H$ of $x,y$, respectively. 
It follows that $w$ is exact, as claimed. 
 \end{proof}

Then using Proposition \ref{injectivestab} we obtain a 
conceptual (without calculation) proof of the following: 

\begin{proposition}
Any homomorphism $f:\pi_1(\Sigma_g)\to S_n$ has a lift 
$f_1:\pi_1(\Sigma_g)\to  B_n^{(1)}$. 
\end{proposition}
\begin{proof}
Let $\mathcal E_1$ denote the extension with abelian kernel: 
\[ 1\to A_1\to B_n^{(1)}\to S_n\to 1\]
whose characteristic class $c_{\mathcal E_1}$ is denoted $e_1\in H^2(S_n,A_1)$. 

Observe first that $f$ admits a lift $f_1$ to $B_n^{(1)}$ if and only if 
the pull-back extension $f^*\mathcal E_1$ admits a section $s$ over $\pi_1(\Sigma_g)$.
This amounts to saying that $f^*\mathcal E_1$ is a split extension which 
is equivalent with $c_{f^*\mathcal E_1}=f^*e_1=0\in H^2(\pi_1(\Sigma_g), A_1)$, where 
$A_1$ has a $\pi_1(\Sigma_g)$-module structure induced by $f$. 

On the other hand Theorem \ref{finitelift} shows that after sufficiently many stabilizations 
$f\circ P_{g,h}$ lifts to a homomorphism $F$ into $B_n$, where $P_{g,h}:\pi_1(\Sigma_{g+h})\to \pi_1(\Sigma_g)$ is the pinch map of the last $h$ handles. 
Let $Q^{(1)}:B_n\to B_n^{(1)}$ be the quotient by $\gamma_1PB_n$. Then $Q^{(1)}\circ F$ is a 
lift of $f\circ P_{g,h}$ to $B_n^{(1)}$. 

The previous argument implies that  $(f\circ P_{g,h})^*\mathcal E_1$ is a split extension over 
$\pi_1(\Sigma_{g+h})$ and hence $P_{g,h}^*\circ f^*e_1=0\in H^2(\pi_1(\Sigma_{g+h}), A_1)$. 
Proposition \ref{injectivestab} implies that $f^*e_1=0\in H_2(\pi_1(\Sigma_g),A_1)$ and thus the claim follows. 
\end{proof}

\begin{remark}
The lifts $f_1$ of $f$ modulo $A_1$-conjugacy are in one-to-one correspondence 
with the section $s$ of $f^*\mathcal E_1$, and thus they form an affine space 
with underlying vector space $H^1(\pi_1(\Sigma_g), A_1)$. It seems possible that  
for any $f$ there exists some lift $f_1$ of $f$ which further can be lifted to $B_n^{(2)}$. 
\end{remark}

\section{Spherical functions and proof of Theorem \ref{separatebraid}}\label{invariants}
\subsection{Pullback spherical functions from Lie groups} 
A key algebraic object in this section is the representation space
\[M_G(\Sigma,B)\subseteq {\rm Hom}(\pi_1(\Sigma\setminus B), G)/G,\]
containing the subspace  $M_G(\Sigma,B, \mathbf c)$  of classes of representations 
with prescribed conjugacy classes of peripheral loops. 
There are analogous moduli spaces of mapping class group orbits: 
\[\M_G(\Sigma, B)=\Gamma(\Sigma\setminus B)\backslash  M_G(\Sigma,B).\] 
We observed in the first section that the corresponding discrete spaces  for 
$G=B_n$ correspond to (strong) equivalence classes of braided surfaces. 
One should note that $G=\Gamma(S)$ corresponds to achiral Lefschetz fibrations with fiber $S$. 
Our aim is to construct functions on these spaces, corresponding in particular to 
invariants of braided surfaces. 
 
In order to treat the unbranched case $B=\emptyset$ we observe that we have a  natural  
embedding:  
\[M_G(\Sigma,\emptyset)\subset {\rm Hom}(\pi_1(\Sigma\setminus \{p\}), G)/G,\]
which provides functions on $M_G(\Sigma,\emptyset)$ by restricting functions defined on the 
right hand side space. 
   
We construct spherical functions on representation spaces associated to discrete groups 
by pullback of spherical functions defined on Lie groups.  
Let $R:G\to \frak G$ be a homomorphism representation of $G$ into the Lie group $\frak G$.  
To any $f:\pi_1(\Sigma\setminus B)\to G$ we associate the homomorphism 
$R_*(f)=R\circ f: \pi_1(\Sigma\setminus B)\to \mathfrak G$.
This induces a map 
\[ R_*: M_G(\Sigma, B) \to M_{\mathfrak G}(\Sigma, B)\]
Obviously the map $R_*$ only depends on the class of $R$ inside 
${\rm Hom}(G, \mathfrak G)/\mathfrak G$. 
Now, the representation variety $M_{\mathfrak G}(\Sigma, B)$ was the subject on 
intensive study, when $\mathfrak G$ is a Lie group.  
 
If $B\neq\emptyset$, then   $M_{G}(\Sigma, B)=G^m/G$, where 
$m$ is the rank of the free group $\pi_1(\Sigma\setminus B)$ and $G$ acts diagonally 
by conjugation  on $G^m$.   
Note that  $M_{G}(\Sigma, B)$ can also be identified with the double 
coset space $G\backslash G^{m+1}/G$, where $G$ is diagonally embedded 
in $G^{m+1}$. 

Let now introduce some terminology from representation theory. 
If $\rho$ is a unitary representation of a group $H$ in a Hilbert space $V$, then a  {\em matrix coefficient} is the function $\phi(x)=\langle \rho(x)v, w\rangle$, where $v,w\in V$. 
Let $L(H)$ be the vector space of complex functions on $H$.  
If $K\subseteq H$ is a subgroup, we denote by 
$L(K\backslash H/K)\subset L(H)$ the subspace of functions which are bi-$K$-invariant, namely 
such that $\phi(k_1 xk_2)=\phi(x)$, for $k_i\in K, x\in H$.  A matrix coefficient 
$\phi(x)=\langle \rho(x)v, w\rangle$ is bi-$K$-invariant if $v, w$ belong to the space 
of $K$-invariants vectors $V^K$. 

Observe first that in the case when $V$ is finite dimensional complex vector space, the  
same formula define a bi-$K$-invariant function, even if $\langle \;,\;\rangle$ is only a Hermitian 
form on $V$, not necessarily positive definite. 
The functions obtained this way will  be called {\em $K$-spherical functions} on $H$; we will add {\em unitary} if we want to specify that the Hermitian form is positive definite.  
Carrying this construction for the pair $K=G$ and $H=G^k$ we  
obtain a family of complex functions on 
$G\backslash  G^k/G$, called spherical functions.
The main question addressed here is to what extent the  spherical functions 
separate points of representation spaces.

\subsection{Compact Lie groups} 
We can organize spherical function by using the map $R_*$ associated to a 
representation of $G$ into some compact group $\mathfrak G$ in order to pullback spherical functions 
from $\mathfrak G^k$. 
The following should be well-known, but for  lack
of references, we sketch the proof: 
\begin{proposition}\label{sphericalseparate}
Let $K\subset H$ be either finite groups or compact connected  Lie groups. Then the unitary  $K$-spherical functions on $H$ separate the points of $K\backslash  H/K$. 
\end{proposition}
\begin{proof}
Let $L(H/K)$ be the Hilbert space of complex valued functions on 
 $H/K$, which is endowed with the tautological left action by $H$. 
 
Consider first the case when $H$ is finite. 
Write $L(H/K)$ as a sum of irreducible representations $V_j$, along with their 
multiplicities $m_j$:  
\[ L(H/K)=\oplus_{j\in J} V_j^{m_j}\]
By Wielandt's lemma (see e.g. \cite{Cecherini}, Thm.3.13.3), the number of 
$K$-orbits in $H/K$ is equal to $\sum_{j\in J} m_j^2$, so that 
\[\dim L(K\backslash H/K)= \sum_{j\in J} m_j^2\]
Now the Frobenius reciprocity (see \cite{Savage}, Thm. 1.4.9) gives us 
\[ m_j =\dim V_j^K\]
Consider the matrix coefficients associated to irreducible representations of 
$\mathfrak G$ into finite dimensional vector spaces $V$ and vectors $v, w$ arising from a basis 
of $V^K$. According to (\cite{Cecherini}, Lemma 3.6.3) matrix coefficients of this form are 
orthonormal in $L(H)$. Since they are $\sum_{j\in J}m_j^2$ elements of $L(K\backslash H/K)$, it follows that they form a basis of $L(K\backslash H/K)$. In particular, the basis functions separate points of $K\backslash H/K)$, as the set of 
all functions in  $L(K\backslash H/K)$ does separate.

The proof in the case where $H$ is a compact Lie group  follows the same lines as in the finite case, now using instead the Peter-Weyl theorem. For instance, matrix coefficients are dense in the space $L^2(G)$, the  
$H$-spherical functions are spanning the space of $L^2$-class functions while $L^2(G)=\oplus W^{\dim W}$ is the direct sum 
of all irreducible representations $W$ of $G$ with multiplicity equal to their dimension. We leave details to the reader.   
\end{proof}

This method provides an infinite family of spherical functions for the case where  $M_{\mathfrak G}(\Sigma, B)=\mathfrak G\backslash \mathfrak G^k/\mathfrak G$, when $B\neq \emptyset$ and 
$\pi_1(\Sigma\setminus B)$ is a free group 
of rank $k-1$. Specifically we consider the set 
$V_i, i\in \widehat{\mathfrak G}$ of all isomorphisms types of irreducible representations of $\mathfrak G$. 
Then $V_{i_1}\otimes \cdots V_{i_k}$ form a representation of $\mathfrak G^k$. 
We should restrict to those unitary representations for which 
$V_{i_1}\otimes \cdots \otimes V_{i_k}$ has a fixed $\mathfrak G$-vector. For each $u,v\in B_I$ in some basis 
$B_I$ of the space of $\mathfrak G$-invariants $H^0(\mathfrak G, V_{i_1}\otimes \cdots \otimes  V_{i_k})$ we have the spherical 
function 
\[ \phi_{u,v, I} (x) = \langle V_{i_1}\otimes \cdots \otimes V_{i_k}(x) u, v\rangle\]
where $I=(i_1,\ldots,i_k)$. 
The (infinite) set of all such functions will separate points of $M_{\mathfrak G}(\Sigma, B)=\mathfrak G\backslash \mathfrak G^k/\mathfrak G$. It is now easy to 
construct a single function taking values in the series in several variables with matrix coefficients: 

\[\Phi(x)= \sum_{I, (u,v)} \frac{1}{I!} (\phi_{I,u,v}(x))_{u,v\in B_I}) X^I\]
A direct consequence of Proposition \ref{sphericalseparate} is:
\begin{proposition}
Assume that $\mathfrak G$ is a compact Lie group.  
Then $\Phi$ is a complete invariant for $M_{\mathfrak G}(\Sigma, B)$, 
namely it separates its points: $\Phi(x)=\Phi(y)$ iff $x=y$.  
\end{proposition}

Neretin (\cite{Ner10}) considered  the case $\mathfrak G=SU(2)$ and expressed (a modified version of) the algebraic function $\Phi$ 
as a determinant. In this case we know that $B_I$ is indexed by the set of partitions $\alpha=(\alpha_{st})_{s,t=1,\ldots,k}$ with 
\[ \sum_{t}\alpha_{st}=i_s\]
Then we consider 
\[ \Phi_N=\sum_{I, (\alpha_{st})} \frac{1}{\alpha!\beta!} \prod_{s,t} {\bf x}^{\alpha}{\bf y}^{\beta}(\phi_{I,\alpha,\beta})\]
where we set ${\bf x}^{\alpha}=\prod_{s,t}x_{st}^{\alpha_{st}}$, $\alpha!=\prod_{s,t}\alpha_{st}!$.  
Then the closed formula  of \cite{Ner10} reads:  
\[ \Phi_N(A) = \det(1-A X A^{\perp} Y)^{-1/2}\]
for $A\in SU(2)^k$, where $X=(X_{ij})$, $Y=(Y_{ij})$  are matrices of blocks of the form  
$X_{ij}=\left(\begin{array}{cc} 0 & x_{ij}\\
-x_{ij} & 0\\
\end{array}\right)$, $Y_{ij}=\left(\begin{array}{cc} 0 & y_{ij}\\
-y_{ij} & 0\\
\end{array}\right)$,
and $x_{ij}, y_{ij}$ are variables. 

In particular we obtain: 
\begin{proposition}
To any representation $R:G\to SU(2)$ of the group $G$ we have associated a polynomial valued invariant 
map $\Phi_R: M_G(\Sigma, B)\to \C[X, Y]$, given by:  
\[ \Phi_R(a)=\det(1- R(A)X R(A)^{\perp} Y).\]
In particular, this holds when the group $G$ is the braid group $B_3$ and $R$ is the Burau representation 
for a parameter within the unit circle $U(1)$, the map $\Phi_{R}$ providing then 
invariants of braided surfaces of degree 3 with nontrivial branch locus. 
\end{proposition}

Often we can reduce the matrix-valued function $\Phi$ to a finite polynomial in more variables. In fact, for any $\mathfrak G$ as above $M_{\mathfrak G}(\Sigma, B)$ is homeomorphic to a finite CW complex. In particular it admits an 
embedding $\varphi:M_{\mathfrak G}(\Sigma, B)\to \R^n$. 
The components of $\varphi$ form therefore a complete 
invariant for $M_{\mathfrak G}(\Sigma, B)$ and so there is a 
much simpler invariant than $\Phi$. Nevertheless 
we lack an exact form of $\varphi$, in general.

In many interesting cases $\mathfrak G\backslash\mathfrak  G^k/\mathfrak G$ has the structure of an (affine) algebraic variety over $\C$. 
Thus we can expect to have a nice algebraic embedding $\varphi$. 
Such an embedding can be obtained from a basis of the algebra of regular functions on 
$M_{\mathfrak G}(\Sigma, B)$. 

This is the case of $\mathfrak G=U(n)$, for instance. 
Let $A=(A_1,A_2,\ldots,A_k)\in U(n)^k$,  $I=\{i_1,i_2,\ldots, i_j\}$, $1\leq i_1 <i_2 <\ldots < i_j\leq k$ and $\varepsilon:I\to \{1,\star\}$. We denote by 
\[A^{I; \varepsilon} = A_{i_1}^{\varepsilon(i_1)}A_{i_2}^{\varepsilon(i_2)}\cdots A_{i_j}^{\varepsilon(i_j)}\]
where $A^\star$ denotes $(A^{-1})^T$. 
Procesi proved in (\cite{Procesi}, Thm. 11.2) that the set of trace functions 
\[ \{ {\rm tr}(A^{I; \varepsilon}) |   \; I\subseteq \{1,2,\ldots,k\}, \; \varepsilon:I\to \{1,\star\} \}\]
over all possible $I$ and $\varepsilon$ represent a basis 
of the algebra of regular functions on $U(n)\backslash U(n)^k/U(n)$. 
Let $x_i$ be noncommutative variables, 
\[ X^{I; \varepsilon}= x_{i_1}^{\overline{\varepsilon}(i_1)}x_{i_2}^{\overline{\varepsilon}(i_2)}\cdots A_{x_j}^{\overline{\varepsilon}(i_j)}\]
where $\overline{\varepsilon}(i)=\varepsilon(i)$, when the later equals $1$ and 
$-1$, otherwise. 
Then the noncommutative Laurent polynomial  
\[ \Psi(A)=\sum_{I, \varepsilon} {\rm tr}(A^{I; \varepsilon}) X^{I; \varepsilon}\]
separates points of $U(n)\backslash U(n)^k/U(n)$.

\begin{proposition}
To any unitary representation $R:G\to U(n)$ of the group $G$ we have associated a 
noncommutative Laurent polynomial valued invariant 
map $\Psi_R$ on $M_G(\Sigma, B)$, given by:  
\[ \Psi_R(a)=\Psi(R(a))\]
In particular, this holds when the group $G=B_n$ and the $R$ is the Burau representation 
for a parameter within the unit circle $U(1)$, the map $\Psi_{R}$ providing then 
invariants of braided surfaces of degree $n$ with nontrivial branch locus. 
\end{proposition}

\begin{remark}
There is a similar result for the noncommutative polynomial  
\[ \Psi'_R(A)=\sum_{I\subseteq \{1,2,\ldots,k\}} {\rm tr}(A^{I}) X^{I}\]
associated to a linear representation  $R:G\to GL(n)$ of the group $G$, which now  
separates points of $GL(n)\backslash GL(n)^k/GL(n)$, following (\cite{Procesi}, section 3). 
\end{remark}

\subsection{Spherical functions for discrete groups} 
Consider now the case of a discrete group $G$.  As observed  above it is enough 
to consider that $B\neq\emptyset$, so that 
$M_G(\Sigma, B)=G\backslash G^k/G$ is a space of cosets. 
In contrast to the case of a compact group $G$, now spherical functions do not 
necessarily separate points of $M_G(\Sigma,B)$. 

For a discrete group $H$ we denote by $\widehat{H}$ its {\em profinite} completion.
There is a natural map $i:H\to \widehat{H}$ which is injective if and only if 
$H$ is residually finite. If $K\subseteq H$ is a subgroup, we denote by $\overline{K}$ the closure 
of $i(K)$ into $\widehat{H}$. The map $i$ induces a map between cosets 
\[ \iota: K\backslash H/K  \to  \overline{K}\backslash \widehat{H}/\overline{K}\]
\begin{definition}\label{separated}
Two cosets of  $K\backslash H/K$  are {\em profinitely separated} if their images 
by $\iota$ are distinct. 
\end{definition}

One case of interest is when $K=G$ is embedded diagonally within $H=G^k$. 
It is easy to see that $\widehat{G^k}$ is isomorphic to $\widehat{G}^k$ and we will identify them in the sequel. If $G$ is embedded diagonally into $G^k$, 
then its closure $\overline{G}$ into $\widehat{G}^k$ is isomorphic to the 
image of $\widehat{G}$ into $\widehat{G}^k$ by its diagonal embedding. Then the map 
$\iota$ from above 
\[ \iota: G\backslash G^k/G  \to  \widehat{G}\backslash \widehat{G}^k/\widehat{G}\]
sends a double coset mod $G$ into its class mod $\widehat{G}$. 
This notion encompasses more classic notions, as the conjugacy separability of 
the group $G$, when we take $k=2$ above.

The main result of this section is: 

\begin{theorem}\label{profinite}
Assume that $H$ is finitely generated and $K\subseteq H$ is a subgroup. 
Two cosets of $K\backslash H/K$ are separated by some Hermitian spherical 
function if and only if they are profinitely separated. 
\end{theorem}
\begin{proof}
Let $x$ and $y$ be cosets which cannot be distinguished by spherical functions 
associated to Hermitian representations of $H$, and in particular by  
spherical functions associated to finite representations. 
Let now $F$ be a finite quotient of $H$, $K_F$ be the image of  $K$ in $F$. 
Proposition \ref{sphericalseparate} shows that spherical functions associated to 
linear representations of $F$ separately precisely the points of $K_F\backslash F/ K_F$.   
Then the images of $x$ and $y$ should coincide in $K_F\backslash F/ K_F$, for any $F$ and 
hence $\iota(x)=\iota(y)$.

Conversely, a finite dimensional Hermitian representation $V$   
of $H$  is defined over some finitely generated ring $\mathcal O\subset \C$. 
By enlarging $\mathcal O$ we can suppose that $\langle \, , \,\rangle$ has entries from $\mathcal O$. 
We can assume, by further enlarging $\mathcal O$, that there is a basis $B$ of  
$V^H$ consisting of vectors whose coordinates belong to $\mathcal O$. 

Suppose that for some $H$-invariant vectors  $u$ and $v$ 
the spherical function $\phi_{u,v}$ separates the cosets $x$ and $y$.  
We can take then $u,v\in \mathcal O \langle B\rangle$. Further, for all but finitely many prime ideals 
$\frak p$ in $\mathcal O$, we have 
\[\phi_{u,v}(x)\not\equiv \phi_{u,v}(y) ({\rm mod} \; \frak p)\in \mathcal O/\mathfrak p.\]
Now let $W$ be the reduction mod $\mathfrak p$  of the $H$ representation on $V$. 
These are finite representations and the invariant subspace 
$W^K$ contains the 
reduction mod $\frak p$ of $V^K$. 
Denote by $\overline{w}$ the reduction mod $\mathfrak p$ of the vector $w\in \mathcal O\langle B\rangle$. It follows that $\overline{u}$ and 
$\overline{v}$ belong to $W^H$. As spherical functions are bilinear, for any $z\in H$ we have: 
\[ \phi_{\overline{u},\overline{v}}(z) \equiv \phi_{u,v}(z) \in \mathcal O/\mathfrak p.\]
In particular, the spherical function  $\phi_{\overline{u},\overline{v}}$ 
associated to a finite representation 
distinguishes $x$ from $y$.  
This implies that $x$ and $y$ are profinitely separated. 
\end{proof}

\begin{remark}
The profinite separability of all cosets in $B_n\backslash B_n^k/B_n$  and their 
mapping class group generalizations seems to be widely open. 
\end{remark}

\subsection{Hurwitz equivalence}
To step from strong equivalence to the usual (i.e. Hurwitz) equivalence amounts of studying 
the action of $\Gamma(\Sigma\setminus B)$ on the vector space of functions on $M_G(\Sigma, B)$.  
However the previous approach using pull-backs of spherical functions from compact Lie groups leads
to a dead end. In fact, we have the following result due to Goldman for $SU(2)$ and 
to Pickrell and Xia for a general compact group: 

\begin{theorem}[\cite{PickrellXia1,PickrellXia2}]
If $\mathfrak G$ is a compact connected Lie group and $\Sigma\setminus B$ is hyperbolic then the action of $\Gamma(\Sigma\setminus B)$ on 
$M_{\mathfrak G}(\Sigma_g, B)$ is ergodic with respect to the quasi-invariant measure. 
\end{theorem}

In particular there are no continuous functions on $M_{\mathfrak G}(\Sigma_g, B)$ 
which are invariant under  
the $\Gamma(\Sigma\setminus B)$ action, other than the constants. Pull-backs of 
spherical functions associated to compact groups could only provide constant functions on 
$\mathcal M_G(\Sigma, B)$. 
In order to get further insight by this method we have to step to non-compact Lie groups and 
the corresponding higher Teichm\"uller theory. As in the previous section, 
components of $\mathcal M_{\mathfrak G}(\Sigma, B)$ which have a CW complex structure, as 
Hitchin components,  will provide functions on the corresponding 
subsets of $\mathcal M_G(\Sigma, B)$.

\bibliographystyle{amsplain}

\begin{thebibliography}{10}

\bibitem{AMS}
I.Agol, G.Meigniez, W.Sawin, 
https://mathoverflow.net/questions/271403

\bibitem{BL}
K. Bou-Rabee and C. Leininger,
{\em Quotients of mapping class groups from $Out(F_n)$},  Proc. Amer. Math. Soc. 146 (2018),  5091--5096.

\bibitem{Brown}
K.Brown, Cohomology of groups, Graduate Texts in Math. 87, 1982. 


\bibitem{CG}
A. Casson and  C. Gordon, {\em Reducing Heegaard splittings}, Topology Appl. 
27(1987), 275--283.

\bibitem{CLP1}
F. Catanese, M. L\"onne and F. Perroni, 
{\em The irreducible components of the moduli space of dihedral covers of algebraic curves}, 
 Groups Geom. Dyn. 9 (2015), 1185--1229.

\bibitem{CLP2}
F. Catanese, M. L\"onne and F. Perroni, 
{\em Genus stabilization for the components of moduli spaces of curves with symmetries},
 Algebraic Geometry  3  (2016),  23--49.

\bibitem{Cecherini}
T.Ceccherini-Silberstein, F. Scarabotti and F. Tolli, Harmonic analysis 
on finite groups, repersentation theory, Gelfand pairs and Markov chains, 
Cambridge Studies Adv. Math., Cambridge Univ. Press. 2008. 


\bibitem{DT}
N.Dunfield and W.P.Thurston, {\em Finite covers of random 3-manifolds}, Inventiones Math. 166 (2006), 457--521. 


\bibitem{Edmonds2}
Allan L. Edmonds, {\em Embedding coverings in bundles}, Canadian Math. Bull. 42(1999), 52--55. 
 
 
\bibitem{Evans}
M. Evans, {\em T-systems of certain finite simple groups}, Math. Proc. Cambridge Phil. Soc. 113 (1993)
9--22.
 
 \bibitem{Evans2}
 M. Evans,{\em  Presentations of groups involving more generators than are necessary}, Proc. London Math.
Soc. 67 (3) (1993) 106--126.
 
\bibitem{F11}
L.Funar, {\em Global classification of isolated singularities in dimensions 
 (4,3) and (8,5)}, Ann. Scuola Norm. Sup. Pisa, vol X (2011), 819--861. 
 
 
\bibitem{FL}
L.Funar and P.Lochak, {\em Profinite completions of Burnside-type surface groups}, 
   Commun. Math. Phys. 360 (2018), 1061--1082. 

\bibitem{Gilman}
R. Gilman, {\em Finite quotients of the automorphism group of a free group}, 
Canad. J. Math. 29 (1977), 541--551. 


\bibitem{Hansen}
V.L. Hansen, 
 {\em  Embedding finite covering spaces into trivial bundles}, Math. Ann. 236, 3 (1978), 239-243

\bibitem{Hansen2}
V.L. Hansen, 
 {\em Polynomial covering spaces and homomorphisms into the braid groups}, 
Pacific J. Math.  81 (1979),  399--410. 

\bibitem{HMM}
G.Hector, G.Meigniez, Y.Matsumoto, 
{\em Ends of leaves of Lie foliations}, J. Math. Soc. Japan 57 (2005), no. 3, 753--779


\bibitem{Hensel}
S. Hensel, {\em A primer on handlebody groups}, 
Handbook of Group Actions, to appear. 


\bibitem{Kamada}
S. Kamada, {\em On braid monodromies of non-simple braided surfaces}, Math. Proc. Cambridge Phil. Soc. 
120 (1996), 237--245. 

\bibitem{Kamada2}
S. Kamada, Braid and knot theory in dimension 4, Math Surveys and Monographs 95, AMS, 2002. 

\bibitem{Jaco}
W. Jaco, {\em Heegaard splittings and splitting homomorphisms}, Trans. Amer. Math. Soc. 
144 (1969), 365--379. 

\bibitem{LM}L.Liechti and J. March\'e, {\em Overcommuting pairs in groups and 3-manifolds bounding them}, 
arXiv: 190311418. 

\bibitem{Liv}
C.Livingston, {\em Stabilizing surface symmetries}, Michigan Math. J. 32 (1985), 249-255.
 
 \bibitem{Liv1}
C.Livingston, {\em Inequivalent bordant group actions on surfaces}, Math. Proc. Cambridge Phil. Soc.  99 (1986), 233--238.
 
 
 
\bibitem{Liv2}
C.Livingston, {\em Maps of surface groups to finite groups with no simple loops in the kernel}, 
J. Knot Theory Ramifications 9 (2000), 1029--1036. 

\bibitem{Lubo}
A. Lubotzky, {\em Dynamics of $Aut(F_n)$ actions on group presentations and representations}, 
"Geometry and Group Actions", p.609--643, a collection of papers dedicated to 
Bob Zimmer, edited by B.Farb and D. Fisher, Chicago Univ. Press, 2011. 

\bibitem{Luft}
E. Luft, {\em 
Actions of the homeotopy group
of an orientable 3-dimensional handlebody}, 
Math. Annalen 234 (1978), 279--292. 


\bibitem{McW}
D.McCullough and M. Wanderley, {\em Free actions on handlebodies}, 
J. Pure Appl. Algebra 181 (2003), 85--104. 


\bibitem{Melikhov}
S.A. Melikhov,  {\em Transverse fundamental group and projected embeddings},  Proc. Steklov Inst. Math. (2015) 


\bibitem{Nakamura}
I. Nakamura, {\em Surface links which are coverings over the standard torus}, 
Alg. Geom. Topology 11 (2011), 1497--1540. 

\bibitem{Ner10}
Y. Neretin, {\em On spherical functions on the group $SU(2)\times SU(2)\times SU(2)$}, 
Funct. Analysis and Appl. 44 (2010), 48--54. 


\bibitem{Pagotto1}
P.G. Pagotto, {\em A product on double cosets of $B_{\infty}$}, SIGMA 14 (2018), 134, 18p.

\bibitem{Pagotto2}
P.G. Pagotto, {\em On the coset monoids of braid groups}, PhD Thesis, Univ. Grenoble-Alpes, 2019,  
{\tt https:/\,/tel.archives-ouvertes.fr/tel-02524192} 



\bibitem{Petersen}
P.Petersen, {\em Fatness of covers},  J. Reine Angew. Math. 403 (1990), 154--165





\bibitem{PickrellXia1}
D.Pickrell and E. Xia, {\em Ergodicity of mapping class group on representation varieties I},
Comment. Math. Helvetici  77 (2002), 339--362. 

\bibitem{PickrellXia2}
D.Pickrell and E. Xia, {\em Ergodicity of mapping class group on representation varieties II}, 
Transformation Groups 8 (2003), 397--402. 


\bibitem{Pikaart}
M.Pikaart, {\em Large characteristic subgroups of surface groups not containing any simple loops}, 
J. Knot Theory Ramifications 10 (2001),  529--535. 


\bibitem{Procesi}
C. Procesi, {\em The invariant theory of $n\times n$ matrices}, 
Advances Math. 19 (1976), 306--381. 

\bibitem{Rudolph83}
L. Rudolph, {\em Braided surfaces and Seifert ribbons for closed braids}, Comment. Math. Helv. 58 (1983), 1--37. 

\bibitem{Rudyak}
Y. Rudyak, On Thom spectra, orientability and cobordism, 
Springer Monographs in Math., corr. edition 2008. 



\bibitem{Savage}
A. Savage, Modern group theory, 
{\tt https:/\,/alistairsavage.ca/mat5145/notes/MAT5145-Modern \underline{ }group\underline{ }theory.pdf}

\bibitem{Zie}
H. Zieschang, {\em Alternierende Producte in freien Gruppen}, Abh. Math. Sem. Hamburg, 27 (1964), 13--31. 


\bibitem{Zim}
B. Zimmermann, {\em Surfaces and the second homology of a group}, 
Monatsh. Math. 104 (1987), 247--253. 

\end{thebibliography}

\end{document}